\documentclass[12 pt]{amsart}

\usepackage{enumerate,xypic,mathdots,mathrsfs,verbatim,amssymb,url}
\usepackage{fullpage}
\usepackage{comment,subfig}
\usepackage[table]{xcolor}
\usepackage{float}

\definecolor{sand}{HTML}{CCD3DB}	
\definecolor{or}{HTML}{F58634}
\definecolor{sea}{HTML}{2E7FAB}
\definecolor{lag}{HTML}{F2DBDC}
\definecolor{lg}{gray}{0.85}
\definecolor{dg}{gray}{0.33}

\numberwithin{equation}{section}

\newtheorem{prop}{Proposition}[section]
\newtheorem{theorem}[prop]{Theorem}
\newtheorem{cor}[prop]{Corollary}
\newtheorem{lemma}[prop]{Lemma}

\theoremstyle{definition}
\newtheorem{defn}[prop]{Definition}
\newtheorem{example}[prop]{Example}

\theoremstyle{remark}
\newtheorem{rem}[prop]{Remark}

\newcommand{\A}{\mathscr{A}}

\newcommand{\C}{\mathbb{C}}

\newcommand{\F}{\mathbb{F}}
\newcommand{\G}{\mathcal{G}}

\newcommand{\Irr}{\text{Irr}}

\newcommand{\R}{\mathbb{R}}

\newcommand{\T}{\mathbb{T}}

\newcommand{\Z}{\mathbb{Z}}

\newcommand{\Norm}[1]{\left\Vert #1 \right\Vert}

\newcommand{\caret}{\char`\^}

\renewcommand{\subset}{\subseteq}

\DeclareMathOperator{\Aut}{Aut}

\DeclareMathOperator{\Hom}{Hom}

\DeclareMathOperator{\mult}{mult}

\DeclareMathOperator{\spn}{span}
\DeclareMathOperator{\ran}{ran}
\DeclareMathOperator{\rank}{rank}

\DeclareMathOperator{\tr}{tr}


\begin{document}

\title{Optimal line packings from nonabelian groups}
\author{Joseph W.\ Iverson}
\address{Department of Mathematics, University of Maryland, College Park, MD 20742}
\curraddr{Department of Mathematics, Iowa State University, Ames, IA 50011}
\email{jwi@iastate.edu}
\author{John Jasper}
\address{Department of Mathematics and Statistics, South Dakota State University, Brookings, SD 57006}
\email{john.jasper@sdstate.edu }

\author{Dustin G.\ Mixon}
\address{Department of Mathematics, The Ohio State University,
Columbus, OH 43210}
\email{mixon.23@osu.edu}

\date{\today}
\maketitle

\begin{abstract}
We use group schemes to construct optimal packings of lines through the origin.
In this setting, optimal line packings are naturally characterized using representation theory, which in turn leads to a necessary integrality condition for the existence of equiangular central group frames.
We conclude with an infinite family of optimal line packings using the group schemes associated with certain Suzuki 2-groups, specifically, extensions of Heisenberg groups.
Notably, this is the first known infinite family of equiangular tight frames generated by representations of nonabelian groups.
\end{abstract}


\section{Introduction} \label{sec:intro}

How does one pack $n$ points in complex projective space $\mathbb{C}\mathbf{P}^{m-1}$ so that the minimum distance is maximized?
At first glance, this problem of \textit{optimal line packings} bears some resemblance to the problem of packing points on a sphere.
The latter problem enjoys famous instances such as the kissing number problem (dating back to Newton~\cite{Newton:66}) and the Tammes problem~\cite{Tammes:30}.
The former problem has been the subject of active research since the seminal work of Conway, Hardin and Sloane~\cite{ConwayHS:96} and the identification by Strohmer and Heath~\cite{SH} of its applications to coding and communication.

Given two points $p,q\in\mathbb{C}\mathbf{P}^{m-1}$, let $a,b\in\mathbb{C}^m$ denote unit-norm representatives of the corresponding lines in $\mathbb{C}^m$, respectively.
Then any reasonable metric on $\mathbb{C}\mathbf{P}^{m-1}$ will take the distance between $p$ and $q$ to be some decreasing function of $|\langle a,b\rangle|$.
For example, the so-called \textit{chordal distance} between $p$ and $q$ is given by $\sqrt{1-|\langle a,b\rangle|^2}$.
In pursuit of an optimal line packing, we therefore seek a sequence $\Phi=\{\phi_i\}_{i=1}^n$ of unit vectors in $\mathbb{C}^m$ that minimizes
\[
\mu(\Phi):=\max_{\substack{i,j\in\{1,\ldots,n\}\\i\neq j}}|\langle \phi_i,\phi_j\rangle|.
\]
We refer to $\mu(\Phi)$ as the \textit{coherence} of $\Phi$.

In order to provably minimize coherence, the most successful approach has been to achieve equality in some lower bound.
To this end, the so-called \textit{Welch bound} has been particularly effective (see~\cite{FMTable} for a survey):
\begin{equation}
\label{eq.Welch bound}
\mu(\Phi)
\geq\sqrt{\frac{n-m}{m(n-1)}}.
\end{equation}
This bound has been rediscovered multiple times~\cite{Rankin:56,Wel74,KabatianskyL:78,ConwayHS:96}, in part because its proof is so simple:
Let $\|\cdot\|_{F}$ be the Frobenius norm. Denoting the $m\times n$ matrix $\Phi=[\phi_1\cdots\phi_n]$ by abuse of notation, we have
\begin{equation}
\label{eq.proof of Welch bound}
0
\leq\left\|\Phi\Phi^*-\frac{n}{m}I\right\|_F^2
=\|\Phi^*\Phi\|_F^2-\frac{n^2}{m}
\leq n+n(n-1)\mu(\Phi)^2-\frac{n^2}{m},
\end{equation}
the last inequality follows from bounding the off-diagonal entries of the Gram matrix $\Phi^*\Phi$ by $\mu(\Phi)$.
Rearranging then gives the Welch bound.

Observe that the proof \eqref{eq.proof of Welch bound} of the Welch bound illuminates exactly when equality occurs:
It is necessary and sufficient that there exist scalars $\alpha$ and $\beta$ such that
\begin{equation}
\label{eq.etf defn}
\Phi\Phi^*=\alpha I
\qquad
\text{and}
\qquad
|\langle \phi_i,\phi_j\rangle|=\beta \text{ whenever } i\neq j.
\end{equation}
(Indeed, $\tr(\Phi\Phi^*)=\tr(\Phi^*\Phi)=n$ forces $\alpha=n/m$, and then $\|\Phi^*\Phi\|_F^2=\|\Phi\Phi^*\|_F^2=n^2/m$ forces $\beta$ to equal the Welch bound.)
Ensembles which satisfy the first requirement in \eqref{eq.etf defn} were historically known as \textit{eutactic stars}~\cite{Schlafli:01}, but the recent literature instead refers to them as \textit{tight frames} due to their significance in frame theory~\cite{DS,DaubechiesGM:86}.
Since the second condition in \eqref{eq.etf defn} imposes equiangularity between the vectors, we refer to Welch bound--equality ensembles as \textit{equiangular tight frames} (ETFs).

For this paper, it is convenient to rescale the vectors in a given ETF so that $\Phi\Phi^*=I$, that is, $\alpha=1$.
In doing so, the norms of our frame vectors become
\begin{equation} \label{eq:FrmNrm}
\Norm{ \phi_i }^2 = \frac{m}{n} \qquad (1 \leq i \leq n),
\end{equation}
and the angle $\beta$ becomes a rescaled version of the Welch bound:
\begin{equation} \label{eq:ParWel}
\beta = \sqrt{\frac{m(n-m)}{n^2(n-1)}}.
\end{equation}
Tight frames with $\alpha=1$ are called \textit{Parseval frames}, since they satisfy the following Parseval-type identity:
\[
\sum_{i=1}^n|\langle x,\phi_i\rangle|^2
=\|x\|^2
\qquad
\text{for all } x\in\mathbb{C}^m.
\]
Note that $\Phi$ is a Parseval frame precisely when $\Phi^*\Phi$ is an orthogonal projection.
Furthermore, given an orthogonal projection $P$, one may take its spectral decomposition to find a Parseval frame $\Phi$ such that $\Phi^*\Phi=P$.
As such, the pursuit of ETFs can be recast as finding orthogonal projections with both a constant diagonal and an off-diagonal of constant modulus.

Some of the earliest known examples of ETFs were \textit{harmonic frames} \cite{SH}, in which $\Phi$ is constructed from the character table of an abelian group by collecting rows indexed by a so-called \textit{difference set} \cite{SH,XZG,DF}.
While harmonic ETFs are the ETFs that exhibit abelian symmetry, the pursuit of new ETFs compels one to seek nonabelian symmetry.
Along these lines, \emph{Zauner's conjecture} in quantum information theory suggests the existence of an ETF of $n=m^2$ vectors in $\mathbb{C}^m$ with Heisenberg symmetry for every $m$~\cite{Z,ScottG:10}.

This paper constructs ETFs with nonabelian symmetry through a generalization of the harmonic ETF theory.
While our construction and its relation to harmonic ETFs can be described in purely group-theoretic terms, we will work in the broader arena of \emph{association schemes}, as detailed in Section 2.
Working in this context offers no additional difficulties, and in our opinion, even simplifies matters.
Every association scheme gives rise to a distinguished set of Parseval frames whose Gram matrices lie in its adjacency algebra.
Characterizing ETFs of this form leads to a notion of \emph{hyperdifference sets}, which unify several disparate ETF constructions, such as harmonic ETFs, those involving skew-symmetric conference matrices~\cite{Renes:07,Strohmer:08}, as well as real ETFs with so-called \textit{centroidal symmetry}~\cite{FickusJMPW:15}.
Section~3 then focuses on a particular type of association schemes, namely, \emph{group schemes}.
Here, hyperdifference sets enjoy a characterization in terms of the representations of the underlying group, and the resulting ETFs are known as \textit{central group frames} (introduced in~\cite{VW}).
Furthermore, this characterization leads to a necessary integrality condition on the existence of such hyperdifference sets, which we then use to significantly reduce the search space in certain cases.

The second half of this paper focuses on a new example of the theory developed in Sections~2 and~3.
Specifically, we discover an infinite family of ETFs by studying the group schemes associated with Suzuki 2-groups~\cite{Hig63}.
The reader primarily interested in our new construction of nonabelian group frames can safely start here and look back to Sections~2 and~3 as necessary.
In Section~4, we use character theory to establish the existence of these ETFs, and in Section~5, we explicitly construct the ETFs out of copies of an extension of the Heisenberg group.
This is particularly interesting in light of Zauner's conjecture~\cite{Z,ScottG:10}, though instead of $n=m^2$, our construction takes
\[
m=2^{2k}(2^{2k+1}-1),
\qquad
n=2^{2(2k+1)}
\]
for any positive integer $k$.

\section{Association schemes and hyperdifference sets} \label{sec:assoc}

In the case of abelian group frames, ETFs are in one-to-one correspondence with difference sets. For nonabelian groups there is a notion of difference set, however there is no obvious connection to ETFs arising as group frames from nonabelian groups. In this section we will present a generalization of difference sets which we call \textit{hyperdifference sets}. The theory of hyperdifference sets is naturally presented in the context of association schemes, which has the pleasing consequence of placing harmonic ETFs, real ETFs with centroidal symmetry, and the ETFs in ~\cite{Renes:07} and~\cite{Strohmer:08} within the same theory.

\subsection{Review of association schemes}

\begin{defn}
An \emph{association scheme} is a pair $\mathfrak{X} = (X, \{R_i\}_{i=0}^d)$ consisting of a finite set $X$ and relations $R_i \subset X \times X$, with the following properties:
\begin{enumerate}[(R1)]
\item $\{R_i\}_{i=0}^d$ is a partition of $X \times X$.
\item Some relation, say $R_0$, is the identity $R_0 = \{ (x,x)  : x \in X \}$.
\item For each $i \in \{ 0, \dots, d \}$, there is some $i'$ such that $R_{i'}= \{ (y,x)  : (x,y) \in R_i \}$.
\item For any $i,j,k\in\{0,\dotsc, d\}$, there is a constant $p_{i,j}^k$ such that, for any $(x,y) \in R_k$,
\[ \left|\{z \in X  : (x,z) \in R_i \text{ and } (z,y) \in R_j\}\right| = p_{i,j}^k. \]
These constants are the \emph{intersection numbers} of the scheme.
\item For any $i,j,k \in \{0,\dotsc,d\}$, we have $p_{i,j}^k = p_{j,i}^k$.
\end{enumerate}
We refer to nonidentity relations as \textit{classes}, so that $d$ is the number of classes.
The constant $k_i = p_{i,i'}^0 = p_{i',i}^0$ is called the \emph{valency} of $R_i$.
The scheme is \emph{symmetric} if each $R_i$ is a symmetric relation.
If properties (R1)--(R4) hold but (R5) fails, we have a \emph{non-commutative} association scheme.
\end{defn}

We associate each relation $R_i$ with an \emph{adjacency matrix} $A_i \in M_n(\C)$ whose $(j,k)$ entry is
\[ A_i(j,k) = \begin{cases}
1, & \text{if } (x_j, x_k) \in R_i \\
0, & \text{if } (x_j, x_k) \notin R_i.
\end{cases} \]
The axioms of an association scheme place equivalent conditions on the adjacency matrices:
\begin{enumerate}
\item[(A1)] $A_0 + \dotsb + A_d = J$, the all-ones matrix.
\item[(A2)] One of the adjacency matrices, say $A_0$, is the identity matrix.
\item[(A3)] The set $\{A_0,\dotsc, A_d\}$ is closed under taking transposes.
\item[(A4)] The adjacency matrices span a subalgebra $\A \subset M_n(\C)$, called the \emph{adjacency algebra} or the \emph{Bose-Mesner algebra}.
Specifically, for every $i,j \in \{0,\dotsc,d\}$ there are numbers $p_{i,j}^k$ (necessarily nonnegative integers) such that
\[ A_i A_j = \sum_{k=0}^d p_{i,j}^k A_k. \]
\item[(A5)] The adjacency algebra is commutative.
\end{enumerate}
The \textit{valency} $k_i$ is the constant number of 1's in any row or column of $A_i$.
The scheme is \textit{symmetric} if and only if $A_i^T = A_i$ for every $i$.

\begin{example}\label{ex:AbSch}
Let $G$ be a finite abelian group. Each element $g \in G$ determines a relation
\[ R_g = \{ (h,k) \in G \times G  : hk^{-1} = g \}, \]
and $\mathfrak{X} := (G,\{R_g\}_{g\in G})$ is an association scheme whose adjacency matrices $\{A_g\}_{g\in G}$ describe the translation operators on $L^2(G)$. The adjacency algebra is isomorphic to the group algebra $\C[G]$ through the mapping $A_g \mapsto \delta_g$, the latter being a canonical basis element of $\C[G]$. Every relation in $\mathfrak{X}$ has valency 1, so the scheme is called \emph{thin}. Conversely, given a thin association scheme we see that the set of adjacency matrices forms a group. The association scheme arises from this group as described above.
\end{example}

Non-symmetric association schemes were first axiomatized by Delsarte in his PhD thesis \cite{Del73}. We recommend \cite{BI} for a thorough introduction. 

Let $\mathfrak{X} = (X,\{R_i\}_{i=0}^d)$ be an association scheme with $|X| = n$, and let $A_0,\dotsc,A_d$ be its adjacency matrices. They form a basis for the adjacency algebra $\A$, which is a commutative $*$-algebra. Putting together (A3) and (A5), we see that the adjacency matrices form a commuting set of normal operators. By the spectral theorem, the adjacency algebra has a second basis of mutually orthogonal primitive idempotents $E_0,\dotsc,E_d$, i.e., the set of projections onto the maximal eigenspaces of the adjacency matrices, and hence the maximal eigenspaces of all of $\mathscr{A}$. Since $\spn\{(1,\dotsc,1)\}$ is a maximal eigenspace of $J \in \A$, one of the idempotents is projection onto that space, and we may assume that $E_0 = (1/n)J$.

In Example \ref{ex:AbSch} we saw that the set of adjacency matrices of an association scheme can be thought of as a generalization of an abelian group. We see in the next example that the basis of idempotents plays the role of the dual group.

\begin{example} \label{ex:AbEig}
When our scheme comes from a finite abelian group $G$, the primitive idempotents are indexed by the Pontryagin dual $\widehat{G}$, which consists of all homomorphisms $\alpha \colon G \to \T$ from $G$ to the multiplicative group of unimodular complex numbers.  With each character $\alpha \in \widehat{G}$, we associate a projection
\[ E_\alpha = \frac{1}{|G|} \sum_{g\in G} \alpha(g) A_g \in \A \]
of rank $m_\alpha = 1$.
\end{example}

The complex vector space spanned by $A_0,\dotsc,A_d$ admits another multiplication which plays an important role in our theory. The \emph{Hadamard product} of two matrices $M_1,M_2 \in M_n(\C)$ is obtained by entry-wise multiplication:
\[ (M_1 \circ M_2)(i,j) = M_1(i,j)\cdot M_2(i,j) \qquad (1\leq i,j \leq n). \]
The corresponding involution is entry-wise complex conjugation, which we denote by
\[ \overline{M}(i,j) = \overline{M(i,j)} \qquad (M \in M_n(\C),\, 1 \leq i,j \leq n). \]
The span of $A_0,\dotsc,A_d$ is closed under both of these operations, since $\overline{A_i} = A_i$ and
\begin{equation} \label{eq:AdjIdem}
A_i \circ A_j = \delta_{i,j} A_i \qquad (i,j \in \{0,\dotsc,d\}).
\end{equation}
To distinguish between Hadamard and matrix multiplication, we write $\widehat{\A}$ for the $*$-algebra spanned by $A_0,\dotsc,A_d$ with the Hadamard product and complex conjugation. Its multiplicative identity is $J = n E_0$. The projection matrices $E_0,\dotsc,E_d$ still form a linear basis for $\widehat{\A}$, and they behave well under its involution: for every $i \in \{0,\dotsc,d\}$, there is some $\hat{\imath} \in \{0,\dotsc,d\}$ such that  $E_{\hat{\imath}} = \overline{E_i} = E_i^T$.

The algebras $\A$ and $\widehat{\A}$ are dual in a certain sense; cf.\ \cite[Ch.\ II, \S5]{BI}. Under this duality, the intersection numbers of $\A$ roughly correspond with the \emph{Krein parameters} of $\widehat{\A}$, which are the unique constants $q_{i,j}^k \in \C$ such that
\[ E_i \circ E_j = \frac{1}{n} \sum_{k=0}^d q_{i,j}^k E_k \qquad (0\leq i,j \leq d). \]
The \emph{Krein condition} says that $q_{i,j}^k \geq 0$ for all $i,j,k\in\{0,\dotsc,d\}$ \cite[Theorem 3.8]{BI}. 

Having reviewed the basics of association schemes, we are now ready to discuss their application to the construction of equiangular tight frames.

\subsection{Parseval frames and hyperdifference sets}

If $\mathfrak{X}$ is an association scheme with primitive idempotents $E_{0},\ldots,E_{d}$, then for any $D\subset \{0,\ldots,d\}$ we form the operator
\[\mathcal{G}_{D} = \sum_{j\in D} E_{j}.\]
Since the maximal idempotents are mutually orthogonal, if follows that $\mathcal{G}_{D}$ is a projection, and hence the Gram matrix of some Parseval frame. Set $m_{i} = \rank E_{i}$ and $m_{D} = \rank\mathcal{G}_{D} = \sum_{i\in D}m_{i}$.

Our goal is not just to create tight frames, but equiangular tight frames. Again, looking at the case of abelian groups is instructive.

\begin{example}\label{ex:HarmFrm}
Let $G$ be a finite abelian group, as in Examples~\ref{ex:AbSch} and \ref{ex:AbEig}. For any $D \subset \widehat{G}$, the matrix
\[ \mathcal{G}_D = \sum_{\alpha \in D} E_\alpha \]
is the $G \times G$ matrix with entries
\[ (\mathcal{G}_D)_{g,h} = \frac{1}{|G|} \sum_{\alpha \in D} \alpha(hg^{-1}) = \frac{1}{|G|} \sum_{\alpha \in D} \alpha(h)\overline{\alpha(g)}\qquad (g,h\in G). \]
This is the Gram matrix of the \emph{harmonic frame} (introduced in~\cite{hochwald2000systematic,goyal2001quantized})
\[ \Phi_D = |G|^{-1/2} \left( \alpha(g) \right)_{\alpha \in D,\, g \in G}, \]
which is obtained by extracting the rows indexed by $D$ from the discrete Fourier transform (DFT) matrix for $G$.
In this context, the choices of $D$ that lead to ETFs are called \textit{difference sets}. Note that $\mathcal{G}_{D}$ is the Gram matrix of an ETF if and only if the off-diagonal entries have constant modulus. This is the motivation for the following definition.
\end{example}

\begin{defn} Let $\mathfrak{X}$ be an association scheme, and let $E_{0},\ldots,E_{d}$ be the basis of primitive idempotents for the adjacency algebra $\mathscr{A}$. A set $D\subset \{0,\ldots,d\}$ is called a \textit{hyperdifference set} if the off-diagonal entries of the matrix $\mathcal{G}_{D}$ all have equal modulus.
\end{defn}

The ``hyper'' here comes from the fact that a hyperdifference set is a subset of the primitive idempotents in the dual adjacency algebra of an association scheme. While this set is a linear basis for the algebra $\widehat{\A}$, it also forms a hypergroup. See Remark~\ref{rem:hyp} below.

In the next example we see that a large class of real ETFs come from hyperdifference sets. In particular, ETFs with centroidal symmetry \cite{FickusJMPW:15}, also known as regular ETFs \cite{BGOY:15}, come from $2$-class association schemes related to certain strongly regular graphs.

\begin{example}\label{ex:RealRegETF} A collection of symmetric $v\times v$ matrices $\{I,A_1,A_2\}$ are the adjacency matrices of a $2$-class association scheme if and only if $A_{1}$ and $A_{2}$ are the adjacency matrices of complementary strongly regular graphs on $v$ vertices. Let $A_{1}$ be the adjacency matrix of a strongly regular graph with parameters $v,k,\lambda,$ and $\mu$, that is, $A_{1}^2+(\mu-\lambda)A_{1}+(\mu-k)I = \mu J$. In this case, the primitive idempotents are given by
\[E_{0} = \frac{1}{v}J,\qquad E_{1} = \frac{\lambda_{-}-k-\lambda_{-}v}{v(\lambda_{+}-\lambda_{-})}I + \frac{v-k+\lambda_{-}}{v(\lambda_{+}-\lambda_{-})}A_{1} + \frac{\lambda_{-}-k}{v(\lambda_{+}-\lambda_{-})}A_{2},\qquad E_{2}=I-E_{0}-E_{1}.\]
where $\lambda_{\pm}=\frac{1}{2}\big[(\lambda-\mu)\pm\sqrt{(\lambda-\mu)^2+4(k-\mu)}\big]$. 

It turns out that this association scheme yields a nontrivial hyperdifference set (i.e.\ distinct from $\{0\}$, $\{1,2\}$, and $\{0,1,2\}$) if and only if	 $2k-v$ is either $2\lambda_{+}$ or $2\lambda_{-}$. If $2k-v=2\lambda_{-}$, then $E_{1}$ and $E_{0}+E_{2}=I-E_{1}$ are Gram matrices of ETFs. That is, $\{1\}$ and $\{0,2\}$ are both hyperdifference sets. Alternatively, if $2k-v=2\lambda_{+}$, then $E_{2}$ and $E_{0}+E_{1} = I-E_{2}$ are Gram matrices of ETFs. That is, $\{2\}$ and $\{0,1\}$ are hyperdifference sets.
\end{example}

\begin{rem}

The association scheme approach outlined above goes back to the work of Delsarte, Goethels, and Seidel on real spherical $t$-designs with few angles~\cite{DGS}.
Specifically, let $\Phi = \{ \varphi_i \}_{i=1}^n \in \mathbb{R}^{m \times n}$ be a sequence of real, equiangular, unit-norm vectors, and let $\mu = | \langle \varphi_i, \varphi_j \rangle |$ for $i \neq j$.
If $\Phi$ is a spherical 2-design, then \cite[Theorem~7.4]{DGS} implies that $\Phi^T\Phi$ carries a 2-class association scheme $\mathfrak{X} = \{ I,A_1,A_2 \}$, in the sense that $\Phi^T \Phi = {I + \mu A_1 - \mu A_2}$.
By~\cite[Proposition 1.2]{HP}, $\Phi$ is a spherical 2-design if and only if it is a tight frame and $\sum_{i=1}^n \varphi_i = 0$.
This is the case for the ETFs in Example~\ref{ex:RealRegETF} corresponding to the hyperdifference sets $\{1\}$ and $\{2\}$, by \cite[Theorem~3.1(b)(ii)]{FickusJMPW:15}.
There are many examples, including several infinite families with $n<\frac{m(m+1)}{2}$ \cite{FickusJMPW:15}.

On the other hand, every real ETF $\Phi$ gives rise to a 3-distance spherical 3-design $\widetilde{\Phi} := [\Phi, -\Phi]$, by~\cite[Example~8.3]{DGS}. 
Then the Gram matrix of $\widetilde{\Phi}$ carries a symmetric, 3-class association scheme by~\cite[Theorem~7.4]{DGS}.
Coming from the other direction, one can start with a symmetric, 3-class association scheme and comb through the idempotents in its adjacency algebra to find $\widetilde{\Phi}^T \widetilde{\Phi}$.
Then one can recover $\Phi$ from $\widetilde{\Phi}$ by discarding a vector from each pair $\{ \varphi_i, -\varphi_i\}$.
In this sense, every real ETF arises from a projection in the adjacency algebra of an association scheme.
For more details and an analogue for complex ETFs, we refer to our followup papers~\cite{IJM2,IM:DTI}.
In contrast, the present article focuses exclusively on ETFs $\Phi$ for which $\Phi^*\Phi$ already lies in the adjacency algebra of an association scheme.

Finally, the complex ETFs which are the focus of this article have considerably less overlap with $t$-designs.
If $\Phi \in \mathbb{C}^{m\times n}$ is a sequence of complex equiangular unit-norm vectors, then it is a complex 2-design (in the sense of Roy~\cite{ROY}) if and only if $n=m^2$, by~\cite[Corollary 3.5.15]{ROY}.
One can also consider \emph{projective} 2-designs, but the same obstruction occurs, as any projective 2-design for $\mathbb{C}\mathbf{P}^{m-1}$ requires at least $m^2$ lines~\cite{BaHo}.
By way of comparison, a complex ETF in $\mathbb{C}^m$ can have at \emph{most} $m^2$ vectors \cite[Proposition 3.4]{HP}, and only finitely examples are currently known to saturate this bound.
(This is the subject of Zauner's conjecture~\cite{Z}, which is still unresolved as of this writing.)
Overall, there is very little overlap between complex ETFs and $t$-designs with few angles.
However, given the similarities between these notions, it is perhaps unsurprising that association schemes give rise to many interesting examples of complex ETFs.

\end{rem}

\begin{prop} \label{thm:AsDiff}
Given $D \subset \{0,\dotsc,d\}$, define constants
\[ b_k := \sum_{i,j \in D} q_{i,\hat{\jmath}}^k \qquad (0 \leq k \leq d). \]
Then the following are equivalent:
\begin{enumerate}[(i)]
\item $D$ is a hyperdifference set
\item There are constants $C_1,C_2\geq 0$ such that
\begin{equation} \label{eq:AsDiff1}
\mathcal{G}_{D} \circ \overline{\mathcal{G}_{D}} = C_1 E_0 + C_2 \sum_{k=0}^d E_k.
\end{equation}
\item $b_1 = b_2 = \dotsb = b_d$.
\end{enumerate}
If one of the equivalent conditions (i)--(iii) holds, then
\begin{equation}\label{eq:AsDiff2} C_1 = \frac{m_D(n-m_D)}{n(n-1)} \qquad \text{and} \qquad  b_k = nC_2 = \frac{m_D(m_D-1)}{n-1} \quad \text{for }k\in\{1,\dotsc,d\}. \end{equation}
\end{prop}

The equivalence of $(i)$ and $(ii)$ above shows that our hyperdifference sets are in some sense dual to the association scheme difference sets as introduced by Godsil~\cite{godsil}. In particular, Godsil's difference sets use adjacency matrices, matrix multiplication, and conjugate transposes where we have primitive idempotents, Hadamard multiplication, and entrywise complex conjugation, respectively, in $(ii)$.

\begin{proof}[Proof of Proposition~\ref{thm:AsDiff}] Since $\mathcal{G}_{D}\circ\overline{\mathcal{G}_{D}}$ is the entrywise square of $\mathcal{G}_{D}$, equivalence of $(i)$ and $(ii)$ follows from the observation that $nE_{0}$ is the all-ones matrix, and $\sum_{j=0}^{d}E_{j}$ is the identity matrix.

For the equivalence of $(ii)$ and $(iii)$, first note that
\begin{equation}\label{eq:AsDiff3}\mathcal{G}_{D}\circ\overline{\mathcal{G}_{D}} = \left(\sum_{i\in D} E_{i}\right)\circ\left(\sum_{j\in D} \overline{E_{j}}\right) = \sum_{i,j\in D}\big(E_{i}\circ\overline{E_{j}}\big) = \sum_{i,j\in D}\frac{1}{n}\sum_{k=0}^{d}q_{i,\hat{\jmath}}^{k}\, E_{k} = \frac{1}{n}\sum_{k=0}^{d}b_{k}E_{k}.\end{equation}

Assume $(ii)$ holds. The matrices $\{E_{j}\}_{j=0}^{d}$ are a basis for the adjacency algebra $\A$. Equating coefficients in \eqref{eq:AsDiff1} and \eqref{eq:AsDiff3} we see that $b_{k} = nC_{2}$ for $k\in\{1,\ldots,d\}$, that is, $(iii)$ holds. Finally, if $(iii)$ holds, then from \eqref{eq:AsDiff3} we see that $(ii)$ holds with $C_{2} = b_{k}/n$ for $k\in\{1,\ldots,d\}$.

Finally, assume that $(i)$--$(iii)$ hold. The matrix $\mathcal{G}_{D}$ is an $n\times n$ projection of rank $m_{D}$ with constant diagonal equal to $m_{D}/n$, and hence
\[\frac{m_{D}}{n}=(\mathcal{G}_{D})_{1,1} = (\mathcal{G}_{D}^{2})_{1,1} = (\mathcal{G}_{D})_{1,1}^{2} + \sum_{j=2}^{n}|(\mathcal{G}_{D})_{1,j}|^{2} = \Big(\frac{m_{D}}{n}\Big)^{2} + (n-1)\frac{C_{1}}{n}.\]
Solving for $C_{1}$ and using the observation from \eqref{eq:AsDiff1} that $C_{1}/n+C_{2} = (m_{D}/n)^2$ we obtain \eqref{eq:AsDiff2}.
\end{proof}

While the preceding proof is very simple, Proposition~\ref{thm:AsDiff} contains the following nontrivial correspondence between harmonic ETFs and difference sets.

\begin{example} \label{ex:AbDif}
Let $G$ be an abelian group, as in Examples~\ref{ex:AbSch}, \ref{ex:AbEig}, and \ref{ex:HarmFrm}. The Pontryagin dual $\widehat{G}$ is itself an abelian group under pointwise multiplication, with
\[ (\alpha \beta)(g) = \alpha(g) \beta(g) \quad \text{and} \quad (\alpha^{-1})(g) = \overline{\alpha(g)} \]
for all $\alpha, \beta \in \widehat{G}$ and $g\in G$. The Hadamard product reflects this multiplication on the primitive idempotents, so that
\[ E_\alpha \circ E_\beta = \frac{1}{|G|} E_{\alpha \beta} \quad \text{and} \quad  \overline{E_\alpha} = E_{\alpha^{-1}} = E_{\widehat{\alpha}} \]
for all $\alpha,\beta \in \widehat{G}$. Hence, the mapping $E_\alpha \mapsto
 |G|^{-1}\delta_\alpha$ gives a $*$-algebra isomorphism $\widehat{\A} \cong \C[\widehat{G}]$, and the Krein parameters are given by $q_{\alpha,\beta}^\gamma = \delta_{\alpha \beta, \gamma}$ for all $\alpha,\beta,\gamma \in \widehat{G}$.

Given a subset $A \subset \widehat{G}$, let $\widetilde{A}, \widetilde{A}^{(-1)} \in \C[\widehat{G}]$ be the vectors
\[ \widetilde{A} = \sum_{\alpha \in A} \delta_\alpha\qquad \text{and} \qquad \widetilde{A}^{(-1)} = \sum_{\alpha \in A} \delta_{\alpha^{-1}}. \]
For fixed $D \subset \widehat{G}$, we can read \eqref{eq:AsDiff1} to be the statement
\begin{equation} \label{eq:AbDif}
\widetilde{D} \widetilde{D}^{(-1)} = C_1 \delta_1 + C_2 \widetilde{\widehat{G}}.
\end{equation}
Meanwhile, the constants $\{b_\gamma\}_{\gamma \in \widehat{G}}$ are given by
\[ b_\gamma = \sum_{\alpha,\beta \in D} \delta_{\alpha\beta^{-1}, \gamma} = \left|\{(\alpha, \beta) \in D \times D  : \alpha\beta^{-1} = \gamma\}\right| \qquad (\gamma \in \widehat{G}). \]
Thus, Proposition~\ref{thm:AsDiff} tells us that $\Phi_D$ is an equiangular tight frame if and only if there is a constant $\lambda$ such that every nontrivial element of $\widehat{G}$ can be written as a difference of elements of $D$ in exactly $\lambda$ ways. In other words, $\Phi_D$ is an ETF if and only if $D$ is a \emph{difference set} in $\widehat{G}$. See \cite{SH,XZG,DF}.
By counting, we observe that $\widehat{G}$ admits a difference set of size $m$ only if $|\widehat{G}|-1$ divides $m(m-1)$, the quotient being $\lambda$.
\end{example}

\begin{rem} \label{rem:hyp}
The situation for general commutative association schemes closely mirrors the previous example if we use the language of hypergroups \cite{Wi97}. Briefly, if we define $a_i = k_i^{-1} A_i$, then $K = \{ a_i \}_{i=0}^d$ has the structure of a commutative hypergroup under matrix multiplication, with involution given by adjoints. As in the case of abelian groups, we must look to the dual hypergroup $\widehat{K}$ to construct ETFs. Setting $e_j = m_j^{-1} E_j$, we have $\widehat{K} = \{ e_j \}_{j=0}^d$ under Hadamard multiplication and entry-wise complex conjugation. Its identity element is $e_0$, and its measure algebra, akin to the group algebra $\C[\widehat{G}]$ in Example~\ref{ex:AbDif}, is $\widehat{\A}$. The Haar measure in $\widehat{K}$ is given by $\mu(e_j)= m_j$; in the abelian group case, this was constantly equal to 1. In parallel with Example~\ref{ex:AbDif}, we associate any subset $A \subset \{0,\dotsc,d\}$ with the vectors $\G_A, \overline{\G_A} \in \widehat{\A}$, which satisfy
\[ \G_A = \sum_{j\in A} \mu(e_j) e_j \qquad \text{and} \qquad \overline{\G_A} = \sum_{j\in A} \mu(e_{\hat{\jmath}}) e_{\hat{\jmath}}. \]
These take the places of $\widetilde{A}$ and $\widetilde{A}^{(-1)}$, respectively. Then Proposition~\ref{thm:AsDiff} tells us that $D\subset \{0,\dotsc,d\}$ is a hyperdifference set if and only if there are constants $c_1,c_2 \geq 0$ such that 
\[ \G_D \circ \overline{\G_D} = c_1 e_0 + c_2 \G_{\widehat{K}}. \]
This is the hypergroup version of \eqref{eq:AbDif}.
\end{rem}

\section{Central group frames and the group scheme} \label{sec:grps}

Given any finite group $G$ with conjugacy classes $C_{0},\ldots,C_{d}$ we define an association scheme $\mathfrak{X}(G)$, called the \textit{group scheme}, whose relations are given by
\begin{equation}\label{def:groupscheme} R_i=\{(g,h) \in G \times G : hg^{-1} \in C_i\} \qquad (0 \leq i \leq d), \end{equation}
with valencies $k_i = |C_i|$.

Let $\Irr(G)$ denote the set of irreducible characters of a group $G$. For each $\chi\in\Irr(G)$ let $d_{\chi}:= \chi(1)$ denote its degree, and let $\pi_{\chi}:G\to M_{d_{\chi}}(\C)$ denote a unitary representation having $\chi$ as its trace character.

The set of primitive idempotents for $\mathfrak{X}(G)$ are given by
\begin{equation}\label{GrpIdem}
(E_{\chi})_{g,h} = \frac{d_\chi}{|G|}\chi(g^{-1}h).
\end{equation}
for any $\chi\in\Irr(G)$ \cite[Theorem 10.6.1]{godsil}. From the definition of $\mathcal{G}_{D}$ it follows that
\begin{equation} \label{eq:GrpGrm}
(\mathcal{G}_D)_{g,h} = \frac{1}{|G|} \sum_{\chi\in D} d_{\chi} \chi(g^{-1}h) \qquad (g,h \in G).
\end{equation}

 We regard $M_{d}(\C)$ as a Hilbert space with the Hilbert-Schmidt inner product $\langle A,B\rangle_{HS} = \tr(AB^{\ast})$. For each representation $\pi_{\chi}$ we define another unitary representation $\rho_{\chi}:G\to U(M_{d_{\chi}}(\C))$ by
\[\rho_{\chi}(g)(A) = \pi_{\chi}(g)\cdot A \qquad \left(g\in G,\, A \in M_{d_\chi}(\C) \right).\]

The following proposition gives an explicit description of the frames made by $\mathfrak{X}(G)$. That is, we give an explicit frame with Gram matrix $\mathcal{G}_{D}$ for $D\subset\Irr(G)$. 

\begin{prop}\label{thm:GrpFrm} Let $G$ be a finite group, and let $D\subset \Irr(G)$. Define the sequence $\Phi_{D}: = \{\phi_{g}\}_{g\in G}$ by
\[\phi_{g}: = |G|^{-1/2}\left(\sqrt{d_{\chi}}\pi_{\chi}(g)\right)_{\chi\in D}\in \bigoplus_{\chi\in D}M_{d_{\chi}}(\C) =: \mathcal{H}_{D}. \]
Then, $\Phi_{D}$ is a Parseval frame for the space $\mathcal{H}_{D}$ of dimension $m_{D}:= \sum_{\chi\in D}d_{\chi}^{2}$. The Gram matrix of $\Phi_{D}$ is $\mathcal{G}_{D}$, and $\Phi_{D}$ is the orbit of $\phi_{1}$ under $\rho_{D}: = \bigoplus_{\chi\in D}\rho_{\chi}$.
\end{prop}

\begin{proof} For $g,h\in G$ we begin by computing the $(g,h)$ entry of the Gram matrix of $\Phi_{D}$:
\begin{align*}
\langle \phi_{h},\phi_{g}\rangle & = \frac{1}{|G|}\sum_{\chi\in D}d_{\chi}\langle \pi_{\chi}(h),\pi_{\chi}(g)\rangle_{HS} = \frac{1}{|G|}\sum_{\chi\in D}d_{\chi}\tr(\pi_{\chi}(h)\pi_{\chi}(g)^{\ast})\\
 & = \frac{1}{|G|}\sum_{\chi\in D}d_{\chi}\tr(\pi_{\chi}(g^{-1}h)) = \frac{1}{|G|}\sum_{\chi\in D}d_{\chi}\chi(g^{-1}h) = (\mathcal{G}_{D})_{g,h}.
\end{align*}

Note that the dimension of $\mathcal{H}_{D}$ is $m_{D}$, and using \eqref{eq:GrpGrm} we have
\[\tr(\mathcal{G}_{D}) = \sum_{g\in G}\frac{1}{|G|} \sum_{\chi\in D} d_{\chi} \chi(g^{-1}g) = \sum_{g\in G}\frac{1}{|G|} \sum_{\chi\in D}d_{\chi}^{2} = \sum_{\chi\in D}d_{\chi}^{2}.\]
Since  $\mathcal{G}_{D}$ is a projection, this shows that the rank of $\mathcal{G}_{D}$ is $m_{D}$.

The claim that $\Phi_{D}$ is the orbit of $\phi_{1}$ under $\rho$ follows from 
\[\rho(g)\phi_{1} = |G|^{-1/2}(\sqrt{d_{\chi}}\pi_{\chi}(g)\pi_{\chi}(1))_{\chi\in D} = \phi_{g}.\qedhere\]
\end{proof}

This construction of $\Phi_D$ is a special case of the generalized harmonic frames introduced separately by the first author \cite{I2} and by Thill and Hassibi \cite{TH}. 

Let $\rho\colon G \to U(\C^m)$ be a unitary representation of $G$, and fix a vector $\phi \in \C^m$. If $\Phi = \left\{ \rho(g)\phi \right\}_{g \in G}$ happens to be a frame, we call it a \emph{group frame}. Its properties can be deduced from the associated \textit{function of positive type} $\psi\in L^{2}(G)$ given by $\psi(g) = \langle \phi,\rho(g)\phi\rangle$ for $g\in G$. One of the important consequences of Proposition \ref{thm:GrpFrm} is that a frame made from a group scheme $\mathfrak{X}(G)$ is a group frame, but more can be said. When $\psi$ lies in the center of the convolution algebra $L^2(G)$, that is, the span of $\Irr(G)$, then the frame is called \emph{central}. This terminology is due to Vale and Waldron \cite{VW}. The following corollary recasts \cite[Theorem~5.1]{VW} in the language of association schemes. It says that central group frames are exactly the frames we are constructing from $\mathfrak{X}(G)$.

\begin{cor} \label{cor:CntFrm}
Let $G$ be a finite group. A Parseval frame has Gram matrix $\G_D$ for some $D \subset \Irr(G)$ if and only if it is a central group frame over $G$.
\end{cor}

\begin{proof} Note that for a group frame $\Phi$ over $G$, the entries of the $G\times G$ Gram matrix are determined by the function of positive type $\psi$. Indeed, for $g,h\in G$, the $(g,h)$ entry of the Gram matrix is exactly $\psi(h^{-1}g)$.

If $\G_{D}$ is the Gram matrix of $\Phi$, then $\Phi$ is unitarily equivalent to the group frame $\Phi_{D}$. The function of positive type associated with $\Phi_{D}$ is $\frac{1}{|G|}\sum_{\chi\in D}d_{\chi} \chi(g^{-1})$. This shows that $\Phi_{D}$, and hence $\Phi$, is a central group frame.

Next, assume $\Phi = \{\rho(g)\phi\}_{g\in G}$ is a central group frame with Gram matrix $\G$ and let $\psi$ be the function of positive type associated with $\Phi$. That $\Phi$ is central means that $\psi$ is constant on conjugacy classes. It follows that $\G$ is constant on every relation $R_{i}\subset G\times G$ in $\mathfrak{X}(G)$, that is, $\G$ is in the adjacency algebra $\mathscr{A}$. Since $\Phi$ is a Parseval frame $\G$ is a projection, and hence $\G$ is a sum of primitive idempotents in $\mathscr{A}$. That is, $\G=\G_{D}$ for some $D\subset \Irr(G)$.
\end{proof}

\begin{cor}\label{cor:CharSums} Let $G$ be a finite group of order $n$, and let $D\subset\Irr(G)$. The set $D$ is a hyperdifference set in $\mathfrak{X}(G)$ if and only if
\begin{equation}\label{flat character sums}\left|\sum_{\chi\in D} d_{\chi} \chi(g)\right| = \sqrt{\frac{m_{D}(n-m_{D})}{n-1}}\quad\text{for all }g\in G\setminus\{1\}.\end{equation}
\end{cor}

\begin{proof} Note that the off-diagonal entries of $n\mathcal{G}_{D}$ are exactly $\sum_{\chi\in D}d_{\chi}\chi(g)$ for $g\in G\setminus\{1\}$. From this it follows that $D$ is a hyperdifference set if and only if the the left side of \eqref{flat character sums} equals some constant $C_{1}$. The value of this constant follows from Proposition~\ref{thm:AsDiff}.
\end{proof}

\subsection{An integrality condition for hyperdifference sets}
The Krein parameters for $\mathfrak{X}(G)$ have a useful interpretation in terms of multiplicities and tensor products. For any $\eta,\tau\in\Irr(G)$ we have
\[E_{\eta}\circ E_{\tau} = \frac{1}{|G|}\sum_{\chi\in\Irr(G)} q_{\eta,\tau}^{\chi}E_{\chi}.\]
Multiplying on both sides by $E_{\chi}$ for some fixed $\chi\in\Irr(G)$ we have
\begin{equation}\label{eq:MltKr0}(E_{\eta}\circ E_{\tau})E_{\chi} = \frac{q_{\eta,\tau}^{\chi}}{|G|}E_{\chi}.\end{equation}
From \eqref{GrpIdem} we know the entries of the matrix $E_{\chi}$. In particular, looking at the $(1,1)$-entry of both sides of \eqref{eq:MltKr0} we obtain the first equality in the following identity:
\begin{equation} \label{eq:MltKr}
q_{\eta,\tau}^\chi = \frac{d_\eta d_\tau}{d_\chi}\cdot \frac{1}{|G|} \sum_{g\in G} \eta(g) \tau(g) \overline{\chi(g)} = \frac{d_\eta d_\tau}{d_\chi}\cdot \frac{1}{|G|}  \mult(\pi_\chi, \pi_\eta \otimes \pi_\tau).
\end{equation}
The hypergroup structure on the dual algebra of $\mathfrak{X}(G)$ is then identical to the usual hypergroup structure on $\Irr(G)$. We have a similar interpretation of hyperdifference sets for $\mathfrak{X}(G)$.

\begin{cor} \label{cor:GrpHD}
Let $D\subset \Irr(G)$, and define $\rho_{D} = \bigoplus_{\chi\in D}\rho_{\chi}$. Then, $D$ is a hyperdifference set for $\mathfrak{X}(G)$ if and only if there is a constant $\lambda\geq 0$ such that
\[ \mult(\pi_{\chi}, \rho_D \otimes \overline{\rho_D}) = d_{\chi}\cdot \lambda \]
for every nontrivial irreducible character $\chi\in\Irr(G)$. In that case,
\begin{equation}\label{eq:GrpHD1}\lambda = \frac{m_D(m_D - 1)}{n-1}.\end{equation}
\end{cor}

\begin{proof} For a character $\chi\in\Irr(G)$ and $g\in G$, the trace character of $\rho_{\chi}$ can be calculated using the orthonormal basis of matrix units $e_{s,t}\in M_{d_{\chi}}(\C)$:
\[
\tr(\rho_{\chi}(g)) = \sum_{s,t = 1}^{d_{\chi}}\langle \rho_{\chi}(g)(e_{s,t}),e_{s,t}\rangle_{HS} = \sum_{s,t=1}^{d_{\chi}} \tr(\pi_{\chi}(g) e_{s,t}^{} e_{s,t}^{\ast}) = d_{\chi}\sum_{s=1}^{d_{\chi}} [\pi_{\chi}(g)]_{s,s} = d_{\chi}\chi(g).
\]
This shows that $\rho_{\chi}\cong \pi_{\chi}^{(d_{\chi})}$, and hence $\rho_{D} \cong \bigoplus_{\chi\in D}\pi_{\chi}^{(d_{\chi})}$.

For any $\chi\in\Irr(G)$ we use \eqref{eq:MltKr} to compute
\begin{align*}
\mult(\pi_{\chi},\rho_{D}\otimes\overline{\rho_{D}}) & = \frac{1}{|G|}\sum_{g\in G}\left(\sum_{\tau\in D}d_{\tau}\tau(g)\right)\left(\sum_{\eta\in D}d_{\overline{\eta}}\overline{\eta(g)}\right)\overline{\chi(g)}\\
 & = \sum_{\tau,\eta\in D}d_{\tau}d_{\overline{\eta}}\frac{1}{|G|}\sum_{g\in G}\tau(g)\,\overline{\eta(g)}\,\overline{\chi(g)} = d_{\chi}\sum_{\tau,\eta\in D}q_{\tau,\overline{\eta}}^{\chi}.
\end{align*}
In the notation of Proposition~\ref{thm:AsDiff} the last expression is equal to $d_{\chi}b_{\chi}$. From the same theorem, we see that $D$ is a hyperdifference set if and only if there is some constant $\lambda$ such that $b_{\chi}=\lambda$ for all nontrivial characters $\chi\in\Irr(G)$. Moreover, from \eqref{eq:AsDiff2} we see that \eqref{eq:GrpHD1} holds.\end{proof}

Corollary~\ref{cor:GrpHD} is the nonabelian version of Example~\ref{ex:AbDif}, which gave the correspondence between ETFs and difference sets for abelian groups. Just as in the abelian setting, this interpretation of hyperdifference sets provides a powerful integrality condition for the existence of ETFs produced by group schemes.

\begin{cor} \label{cor:GrpInt}
Let $n\geq 2$, and let $\{\phi_{i}\}_{i=1}^{n}$ be a frame for $\C^{m}$. If $\{\phi_{i}\}_{i=1}^{n}$ is an ETF and a central group frame, then
\[ \frac{ m(m-1) }{n-1} \in \Z. \]
\end{cor}

\begin{proof}
By Corollary~\ref{cor:CntFrm}, we may assume that $n=|G|$ and that the Gram matrix of our ETF is $\G_D$ for some $D \subset \Irr(G)$, which must be a hyperdifference set. By Corollary~\ref{cor:GrpHD}, for any nontrivial character $\chi\in\Irr(G)$
\[ d_\chi \cdot \frac{ m(m-1) }{n-1} \in \Z.\]
In other words, $n-1$ divides $d_\chi \cdot m(m-1)$. A basic result of character theory says that $d_\chi$ divides $n = |G|$. Therefore, $d_\chi$ and $n-1$ are coprime. Thus, $n-1$ divides $m(m-1)$.
\end{proof}

\subsection{Hyperdifference sets of constant degree}\label{sect:HDSCD}

We now apply the integrality condition of Corollary~\ref{cor:GrpInt} for a special class of indexing sets. Thill and Hassibi \cite{TH} have suggested a construction for low-coherence frames that always produces a central group frame. In the abelian case, their construction sometimes yields actual ETFs. In the following section, we will produce an infinite family of ETFs generated by nonabelian groups, each of which is an instance of the construction in \cite{TH}. The frames of \cite{TH} have a useful feature: their indexing sets $D\subset \Irr(G)$ are all characters of equal degree. For this class of frames, the integrality constraint of Corollary~\ref{cor:GrpInt} is especially discriminating.

We say that a hyperdifference set $D\subset\Irr(G)$ is of constant degree with parameters $(n,k,l,m)$ if $|G|=n$, $|D|=k$, $d_{\chi} = l$ for each $\chi\in D$, and $m=m_{D}$. The goal of this subsection is to produce a short list of possible parameters of hyperdifference sets of constant degree in nonabelian groups of order $n<1024$.  To this end we will combine the integrality condition of Corollary~\ref{cor:GrpInt}, a few group-theoretic lemmas outlined below, and an exhaustive computer search using GAP \cite{GAP} and Sage \cite{sage}. For the sake of reproducibility, our code is available in \cite{GitHub}.

In this section we will find use for the following notation: $\Irr_{l}(G)$ is the set of irreducible characters of degree equal to $l$, and $\Irr_{>1}(G)$ is the set of nonlinear irreducible characters. We will suppress the group if there will be no confusion.

\begin{lemma} \label{lem:kg1}
Let $G$ be a nonabelian group. If $D\subset\Irr(G)$ is a hyperdifference set of constant degree with parameters $(n,k,l,m)$ and $m\neq 1$, then $l \geq 2$.
\end{lemma}

\begin{proof}
Every character of degree 1 is constantly equal to $1$ on the commutator subgroup $[G,G]$, which must be nontrivial since $G$ is nonabelian. If $l=1$, then any $g \in [G,G] \setminus\{1\}$ satisfies
\[ \left| \sum_{\chi\in D} d_\chi \chi(g) \right| = m_D \neq \sqrt{ \frac{ m_D(n-m_D) }{n-1} }, \]
contrary to Corollary~\ref{cor:CharSums}.
\end{proof}

\begin{lemma} \label{lem:ConjSz}
Suppose there is a hyperdifference set $D\neq \emptyset$ for $\mathfrak{X}(G)$ such that $d_\chi \geq 2$ for all $\chi\in D$. If $d_M  = \max_{\chi\in D} d_\chi$, then the size of each of the $d$ conjugacy classes $C_i\subset G$ other than $C_0 = \{1\}$ satisfies the inequality
\begin{equation} \label{eq:ConjSz}
\frac{|G|}{|C_i|} \geq \frac{|G|}{|[G,G]|} + \frac{ m_D}{d_M^2|D|}\cdot \frac{|G|-m_D}{|G|-1} \qquad (1 \leq i \leq d),
\end{equation}
Consequently:
\begin{enumerate}[(i)]
\item Every coset of $[G,G]$ contains at least two conjugacy classes.
\item At least half of the irreducible characters of $G$ are nonlinear.
\end{enumerate}
\end{lemma}

\begin{proof}
Let $C_i\subset G$ be a conjugacy class other than $C_0 = \{1\}$, and let $g \in C_i$ be arbitrary. By the column orthogonality relations \cite[Theorem 16.4]{JL},
\[ \frac{|G|}{|C_i|} = \sum_{\chi\in\Irr} | \chi(g)|^2 = \sum_{\chi \in \Irr_1} | \chi_(g)|^2 + \sum_{\chi \in \Irr_{>1}} | \chi(g)|^2. \]
Since $|\chi(g)| = 1$ for every $\chi \in \Irr_1$, and since $D\subset \Irr_{>1}$,
\begin{equation} \label{eq:ConjSz2}
\frac{|G|}{|C_i|} \geq |\Irr_1(G)| + \sum_{\chi\in D} | \chi(g) |^2 = \frac{|G|}{|[G,G]|} + \sum_{\chi\in D} | \chi(g) |^2.
\end{equation}
On the other hand, Corollary~\ref{cor:CharSums} and the Cauchy--Schwarz inequality show that
\begin{equation} \label{eq:ConjSz3}
\frac{m_D(|G|-m_D)}{|G|-1} = \left| \sum_{\chi\in D} d_\chi \chi(g) \right|^2 \leq  \left( \sum_{\chi \in D} d_\chi^2 \right) \left( \sum_{\chi \in D} | \chi(g) |^2 \right) \leq d_M^2 |D| \sum_{\chi\in D} | \chi(g) |^2.
\end{equation}
Combining \eqref{eq:ConjSz2} and \eqref{eq:ConjSz3} gives \eqref{eq:ConjSz}.

Since $G/[G,G]$ is abelian, every conjugacy class $C_i$ lies in a single coset of $[G,G]$, so that
\[ |C_i| \leq |[G,G]|. \]
Equality obviously fails when $i=0$, since $G$ has a nonlinear character and is therefore nonabelian. It cannot hold for any $i\geq 1$, either, or else \eqref{eq:ConjSz} will fail. Thus, the coset of $[G,G]$ containing $C_i$ contains another conjugacy class, too. This is the case for every conjugacy class, hence for every coset of $[G,G]$. This proves (i).

For (ii), recall that the number of irreducible characters equals the number of conjugacy classes. By (i),
\[ d+1 \geq 2(G : [G,G]) = 2|\Irr_1|. \]
In other words, no more than half of the irreducible characters are linear.
\end{proof}

 The remainder of this section will be dedicated to showing that the parameters of any hyperdifference set of constant degree in a nonabelian group of order $n<1024$ with $m\notin\{0,1,n\}$ appear in Table \ref{table:int}.

\begin{table}
\begin{tabular}{rrrr|r}
$n$ & $k$& $l$ & $m$ & \# \\ \hline\hline
64 & 7 & \hphantom{00}2 & 28 & 10 \\
256 & 30 & 2 & 120 & 1936 \\
256 & 34 & 2 & 136 & 1936 \\
320 & 22 & 2 & 88 & 17 \\
320 & 58 & 2 & 232 & 10 \\
576 & 69 & 2 & 276 & 56 \\
576 & 75 & 2 & 300 & 56\\
640 & 18 & 2 & 72 & 799 \\
896 & 45 & 2 & 180 & 709 \\
896 & 179 & 2 & 716 & 41
\end{tabular}
\vspace{12pt}
\caption{The parameters for which there may exist hyperdifference sets of constant degree in nonabelian groups of orders less than 1024.
The column labeled ``\#'' gives the number of groups of order $n$ that might admit a hyperdifference set with parameters $(n,k,l,m)$; see \cite{GitHub} for a complete list.}
\label{table:int}
\end{table}

Suppose we have a hyperdifference set $D$ of constant degree with parameters $(n,k,l,m)$ in a nonabelian group $G$ of order less than $1024$. The constant degree $l = d_\chi$ must divide $n = |G|$ and satisfy
\[ kl^2 = m < n, \]
since $m$ is the dimension of a frame with $n$ vectors. Moreover, $l\geq 2$ by Lemma~\ref{lem:kg1}. We performed a brute force search for all parameters $(n,k,l,m)$ with these properties that also satisfy the integrality condition of Corollary~\ref{cor:GrpInt}. This produced 238 tuples $(n,k,l,m)$.

The computer program GAP \cite{GAP} contains a library of all groups with order less than 1024. For each of our 238 tuples, we used GAP to determine whether or not there was a nonabelian group of order $n$ with conjugacy classes of sizes consistent with Lemma~\ref{lem:ConjSz}. That reduced the list to 38 tuples.

For the groups that remained, we again used GAP to compute their full character tables, and checked two more conditions. First, the group had to actually possess $k$ characters of degree $l$. Second, Corollary~\ref{cor:CharSums} required that
\begin{equation} \label{eq:ConstDeg2}
\frac{m(n-m)}{n-1} = l^2 \left| \sum_{\chi\in D} \chi(g) \right|^2 \leq kl^2  \sum_{\chi\in D} | \chi(g) |^2 \leq m \sum_{\chi \in \Irr_l} | \chi(g)|^2
\end{equation}
for all $g \neq 1$. In other words, the character table of $G$ had to satisfy
\[ \sum_{\chi\in \Irr_l} |\chi(g)|^2 \geq \frac{n-m}{n-1} \qquad (g \neq 1). \]
We used the package FUtil \cite{futil} to make the necessary comparisons in GAP. Only 11 of the remaining 38 tuples had groups that passed this test. One of these, $(64,9,2,36)$, was small enough to allow a brute force check on all ten of its candidate groups. We used GAP to verify that no such hyperdifference set exists. The ten tuples that remained appear in Table~\ref{table:int}.

To pare down the list of candidate groups even further, we turned to Sage \cite{sage}. The current implementation of cyclotomic numbers in GAP does not allow for square roots of irrational real numbers. Consequently, it can compute things like $|z|^2 = z\overline{z}$, but not $|z| = \sqrt{z\overline{z}}$. That is why we used the Cauchy--Schwarz inequality in \eqref{eq:ConstDeg2}, even though the triangle inequality gives a potentially tighter bound:
\begin{equation} \label{eq:ConstDeg3}
\sqrt{\frac{m(n-m)}{n-1}} = l \left| \sum_{\chi\in D} \chi(g) \right| \leq l \sum_{\chi\in D} | \chi(g) | \leq l \sum_{\chi\in \Irr_l} | \chi(g)| \qquad (g\neq 1).
\end{equation}
Unlike GAP, Sage has an implementation of algebraic numbers that allows the user to take absolute values, while still affording exact computations. Using Sage's internal interface to GAP, we checked the groups that passed all of our previous tests for compliance with the inequality
\begin{equation} \label{eq:ConstDeg4}
\sum_{\chi\in \Irr_l} | \chi(g) | \geq \sqrt{\frac{k(n-m)}{n-1}} \qquad (g\neq 1),
\end{equation}
which is equivalent to \eqref{eq:ConstDeg3}. Where equality held, we checked additional necessary conditions. Namely, if equality holds across \eqref{eq:ConstDeg3} for some $g$, then $\chi(g) = 0$ for all $\chi\in \Irr_l\setminus D$. Moreover, the nonzero elements of $\{\chi(g)\}_{\chi\in D}$ must all have the same phase, since equality held in the triangle inequality. The same is then true of $\{\chi(g)\}_{\chi\in \Irr_l}$. Finally, the set
\[ \widetilde{D} = \{ \chi \in \Irr_l \colon \chi (g) \neq 0 \text{ for some $g\neq 1$ that produces equality in \eqref{eq:ConstDeg4}}\} \subset D \]
must have cardinality $|\widetilde{D}| \leq k$, and if equality holds then $\widetilde{D}$ must actually be a hyperdifference set. The number of groups that passed all of these tests appears in the table. This completes the argument that the table above is complete.

When a hyperdifference set with parameters in our table exists, it produces an $m \times n$ ETF. It turns that all ten of the groups implicated by the first row of the table have hyperdifference sets with parameters $(64,7,2,28)$.  In the next section, we will prove that one of these is the smallest example in a new infinite family of hyperdifference sets for group schemes with parameters $n=2^{4j+2}$, $k=2^{2j+1}-1$, and $l=2^j$ for $j\geq 1$. We are especially interested in the last two rows of the table, because if there are hyperdifference sets with these parameters, they produce previously unobserved ETFs \cite{FMTable}.

\begin{rem}
In the next section, we build an infinite family of nonabelian groups with hyperdifference sets of a form suggested by \cite{TH}: For some subgroup $A \subset \Aut(G)$, $D$ is an orbit of the action of $A$ on $\Irr(G)$. Any hyperdifference set of this form consists of characters with equal degrees, so if it exists in a group of order $n < 1024$, its parameters appear in the table. Meanwhile, $k$ must divide $|A|$, which must divide $|\Aut(G)|$. We made an exhaustive search of all the parameters and groups that passed the tests described in this subsection, and the only one for which $k$ divides $|\Aut(G)|$ is SmallGroup(64,82), with $k = 7$. That group really does have a hyperdifference set of 7 characters with the form suggested in \cite{TH}. However, no other nonabelian group of order $n < 1024$ possesses a hyperdifference set with this form.
\end{rem}

\section{Hyperdifference sets for Suzuki 2-groups}

For the remainder of the paper, we focus on an example of the theory built so far. In this section, we construct a new infinite family of hyperdifference sets for group schemes. This is the first time an infinite family of nonabelian groups has been shown to generate ETFs as projective orbits under unitary representations. As we will see in Section~\ref{sec:Heis}, the frames we build here are intimately connected with finite Heisenberg groups. 

\begin{defn}
Let $U$ and $V$ be vector spaces over the same field $\mathbb{K}$, and let $B \colon U \times U \to V$ be a $\mathbb{K}$-bilinear map. We write $U \times_B V$ for the group with underlying set $U \times V$ and multiplication
\[ (u,v) \cdot (x,y) = (u+x, v+y+B(u,x)) \qquad (u,x\in U;\, v,y \in V); \]
we call $U \times_B V$ the \emph{$B$-product} of $U$ and $V$.
\end{defn}

The reader can verify that $U\times_B V$ forms a group with identity $(0,0)$ and inverses given by
\[ (u,v)^{-1} = (-u,-v+B(u,u)) \qquad (u\in U,\, v \in V). \]
In fact, $U\times_B V$ is the central extension of $(U,+)$ by $(V,+)$ associated with the 2-cocycle $B$.

\begin{example}
When $U=\R^n$, $V=\R$, and $B$ is the usual dot product, $U \times_B V$ is the Heisenberg group $\mathbb{H}_n$.
\end{example}

In general, $U\times_B V$ is not abelian unless $B$ is symmetric. In fact, much of the structure of $U \times_B V$ is controlled by the antisymmetric bilinear map $\widehat{B} \colon U \times U \to V$,
\[ \widehat{B}(u,v) = B(u,v) - B(v,u) \qquad (u,v \in U), \]
which acts like a sort of commutator. For each $u \in U$, we let $L_u \colon U \to V$ be the linear map
\[ L_u(v) = \widehat{B}(u,v) = B(u,v) - B(v,u) \qquad (v \in U), \]
and we let $\mathcal{L} \colon U \to \Hom_{\mathbb{K}}(U,V)$ be the linear function with $\mathcal{L}(u) = L_u$.

\begin{prop} \label{prop:Bprods}
Keep notation as above.
\begin{enumerate}[(i)]
\item $Z(U\times_B V) = (\ker \mathcal{L}) \times V = \{ u \in U \colon \widehat{B}(u,v) = 0 \text{ for all }v \in U\} \times V$. \smallskip
\item $[U\times_B V, U \times_B V] = \{0\} \times \spn_{\mathbb{K}}\{ \widehat{B}(u,v) \colon u,v \in U\}$. \smallskip
\item Conjugacy classes in $U\times_B V$ take the form
\[ (u,v)^{U\times_B V} = \{ (u,v+\widehat{B}(u,w)) \colon w \in U\} \qquad (u\in U,\, v \in V). \]
\item If $W\subset V$ is a vector subspace, and if $q\colon V \to V/W$ is the natural quotient, then there is a group homomorphism $\varphi \colon U \times_B V \to U \times_{q\circ B} (V/W)$ given by $\varphi(u,v) = (u,v+W)$ for $u\in U$ and $v\in V$. It factors to give an isomorphism
\[ (U\times_B V)/(\{0\}\times W) \cong U \times_{q \circ B} (V/W). \]
\end{enumerate}
\end{prop}

\begin{proof}
Let $u,x\in U$ and $v,y \in V$. The reader can check that
\[ (x,y)^{-1} \cdot (u,v) \cdot (x,y) = (u,v+ \widehat{B}(u,x)) = (u,v+ L_u(x)), \]
hence
\[ (u,v)^{-1}\cdot (x,y)^{-1}\cdot (u,v) \cdot (x,y) = (0,\widehat{B}(u,x)). \]
Statements (i)--(iii) follow immediately. The reader can check that $\varphi$ is a homomorphism in (iv); the rest is immediate.
\end{proof}

\begin{defn}
Use notation as above. We will say that $B$ has the \emph{injective hyperplane property} (IHP) if:
\begin{enumerate}[(i)]
\item For each $u \in U \setminus\{0\}$, $\ran L_u$ is a hyperplane in $V$. That is, $\dim_{\mathbb{K}} V/(\ran L_u) = 1$.
\item The function $u\mapsto \ran L_u$ is injective on $U\setminus\{0\}$.
\end{enumerate}
\end{defn}

For the remainder of this section, we assume that $U=V$ is a vector space over $\mathbb{K} = \F_2$ with odd dimension $n = 2k+1 \geq 3$, and that $B \colon U \times U \to U$ is a bilinear map with IHP. We will prove that $G = U \times_B U$ admits a nontrivial hyperdifference set.

For each $u \in U \setminus\{0\}$, we let
\[ H_u = \ran L_u = \{ \widehat{B}(u,v) \colon v \in U \} \]
be the hyperplane guaranteed by IHP. Then $U /H_u \cong \F_2$ in exactly one way, and with this identification in mind we let $q_u \colon U \to \F_2$ give the natural quotient of $U$ onto $U/H_u$; explicitly,
\[ q_u(v) = \begin{cases}
0, & \text{if }v \in H_u \\
1, & \text{otherwise}
\end{cases} \qquad (v\in U). \]
We define $G_u = U \times_{q_u \circ B} \F_2$, whose multiplication is given by
\[ (x,\epsilon)\cdot (y,\delta) = \begin{cases}
(x+y,\epsilon + \delta), & \text{if }B(x,y) \in H_u \\
(x+y,\epsilon + \delta + 1), & \text{otherwise}
\end{cases} \qquad (x,y\in U;\, \epsilon,\delta \in \F_2). \]
By Proposition~\ref{prop:Bprods}(iv), $G_u$ is essentially a copy of $G / (\{0\} \times H_u)$. More precisely, we have a surjective homomorphism $\varphi_u \colon G \to G_u$ given by
\[ \varphi_u(x,y) = \begin{cases}
(x,0), & \text{if }y \in H_u \\
(x,1),  & \text{otherwise}
\end{cases} \qquad (x,y \in U). \]
Our plan is to pull a nonlinear character of $G_u$ back to $G$ through $\varphi_u$. As $u$ iterates through $U \setminus\{0\}$, these pullbacks will form a hyperdifference set for $G$.

\begin{theorem} \label{thm:IHP}
For each $u \in U \setminus\{0\}$, $G_u$ has exactly two nonlinear irreducible characters, both of dimension $2^k$. If $\chi_u \colon G_u \to \C$ is either choice of these two, and if $\widetilde{\chi}_u = \chi_u \circ \varphi_u$ is its pullback to $G$, then $\{ \widetilde{\chi}_u \}_{u\in U \setminus\{0\}}$ is a hyperdifference set for $G$.
\end{theorem}

\begin{proof}
We begin by computing some basic properties of $G_u$ with Proposition~\ref{prop:Bprods}. For any $x,y\in U$,
\[ (q_u \circ B)\caret(x,y) = q_u(\widehat{B}(x,y)) = \begin{cases}
0, &\text{if }\widehat{B}(x,y) \in H_u \\
1, & \text{otherwise}
\end{cases}\qquad (x,y \in U). \]
When $x = u$, we clearly get $(q_u\circ B)\caret(u,y) = 0$ for all $y \in U$. On the other hand, if $x \notin \{0,u\}$ then $H_x \neq H_u$, so there is some $y \in U$ with $(q_u \circ B)\caret(x,y) \neq 0$. Thus,
\begin{equation} \label{eq:IHP1}
Z(G_u) = \{0,u\} \times \F_2 \qquad \text{and} \qquad [G_u,G_u] = \{0\} \times \F_2.
\end{equation}
Moreover, for any $x\in U$ and $\epsilon \in \F_2$,  the conjugacy class of $(x,\epsilon)$ is
\[ (x,\epsilon)^{G_u} = \begin{cases}
\{(x,\epsilon)\}, & \text{if }x \in\{0,u\} \\
\{(x,0),\, (x,1)\}, & \text{otherwise.}
\end{cases} \]

Next, we find the degrees of the irreducible characters of $G_u$. All together, there are $2^n + 2$ conjugacy classes in $G_u$, for $2^n + 2$ irreducible characters. Of these,
\[ \frac{|G_u|}{|[G_u,G_u]|} = 2^n \]
have degree 1. We must show that the other two have degree $2^k$. Their degrees must divide $|G_u| = 2^{n+1}$, so we may assume they are $2^{k_1}$ and $2^{k_2}$ with $k_1 \leq k_2$. Since the squares of all the degrees of irreducible characters sum to $|G_u|$, we must have
\[ 2^n + 2^{2k_1} + 2^{2k_2} = 2^{n+1}, \]
or equivalently,
\begin{equation} \label{eq:IHP7}
1 + 2^{2k_2-2k_1} = 2^{n-2k_1}.
\end{equation}
Since $n$ is odd, the right-hand side is even. For this to be true of the left-hand side, we must have $2k_1 - 2k_2 = 0$. Then another look at \eqref{eq:IHP7} shows that $k_1 = k_2 = k$.

Let $\chi_u$ be one of the two irreducible characters of $G_u$ with degree $2^k$. We claim that
\begin{equation} \label{eq:IHP2}
\chi_u(x,\epsilon) = 0 \text{ for all $x \notin \{0,u\}$ and all $\epsilon \in \F_2$,}
\end{equation}
and that
\begin{equation} \label{eq:IHP3}
(0,1) \notin \ker \chi_u.
\end{equation}
For an irreducible representation corresponding to $\chi_{u}$, Schur's Lemma implies that $Z(G_{u})$ acts by scalar multiples of the identity. Hence, $|\chi_u(x,\epsilon)| = 2^k$ for all $(x,\epsilon) \in Z(G_u)$. The row orthogonality relations require that
\[ |G_u| = \sum_{(x,\epsilon) \in Z(G_u)} |\chi_u(x,\epsilon)|^2 + \sum_{(x,\epsilon) \notin Z(G_u)} | \chi_u(x,\epsilon)|^2, \]
so
\[ 2^{n+1} = 4\cdot 2^{2k} + \sum_{(x,\epsilon) \notin Z(G_u)} | \chi_u(x,\epsilon)|^2. \]
Since $4\cdot 2^{2k}$ already equals $2^{n+1}$, we conclude that $\chi_u(x,\epsilon) = 0$ whenever $(x,\epsilon) \notin Z(G_u)$. This proves \eqref{eq:IHP2}. Meanwhile, if $(0,1) \in \ker \chi_u$, then $\chi_u$ factors to give an irreducible character of the abelian group $G_u / [G_u,G_u]$, despite the fact that its degree is $2^k \neq 1$. (Here we use the fact that $n \geq 3$.) This establishes \eqref{eq:IHP3}.

Finally, let $\widetilde{\chi}_u$ be the pullback of $\chi_u$ to $G$ via $\varphi_u$. That is,
\begin{equation} \label{eq:IHP4}
\widetilde{\chi}_u(x,y) = \begin{cases}
\chi_u(x,0), & \text{if }y \in H_u \\
\chi_u(x,1), & \text{otherwise.}
\end{cases}
\end{equation}
We will prove that $\{\widetilde{\chi}_u\}_{u\in U \setminus \{0\}}$ is a hyperdifference set, and in particular that
\begin{equation} \label{eq:IHP5}
\left| \sum_{u \in U \setminus\{0\}} \widetilde{\chi}_u(x,y) \right| = 2^k \text{ for all $(x,y) \neq (0,0)$ in $G$.}
\end{equation}

If $x \in U\setminus\{0\}$, then \eqref{eq:IHP2} shows that $\widetilde{\chi}_u(x,y) = 0$ for all $u\neq x$. On the other hand, $\varphi_x(x,y) \in Z(G_x)$ by \eqref{eq:IHP1}, so $|\chi_x(x,y)| = 2^k$. Thus, \eqref{eq:IHP5} holds whenever $x\neq 0$.

It remains to prove \eqref{eq:IHP5} in the case where $x=0$. This is exactly the case where $(x,y) \in Z(G)$, by Proposition~\ref{prop:Bprods}: since $\ran L_u = H_u \neq \{0\}$ for all $u \in U\setminus\{0\}$, $Z(G) = \{0\} \times U$. In particular, $Z(G)$ has exactly $2^n$ distinct linear characters. For each $u \in U \setminus\{0\}$, let $\alpha_u \colon Z(G) \to \T$ be the central character of $\widetilde{\chi}_u$; in other words, $\widetilde{\chi}_u(0,y) = \alpha_u(0,y)\cdot 2^k$. From \eqref{eq:IHP4} and \eqref{eq:IHP3}, we see that
\[ \widetilde{\chi}_u(0,y) = \begin{cases}
\chi_u(0,0) = 2^k, & \text{if }y \in H_u \\
\chi_u(0,1) \neq 2^k, & \text{otherwise.}
\end{cases} \]
Therefore, $\ker \alpha_u = \{0\} \times H_u$.

By the injective hyperplane property, the characters $\alpha_u$ are all different. In particular, they exhaust all $2^n-1$ nontrivial characters of $Z(G)$. Now the column orthogonality relations for $Z(G)$ show that
\begin{equation} \label{eq:IHP6}
\sum_{u \in U\setminus\{0\}} \widetilde{\chi_u}(0,y) = 2^k\cdot \sum_{u \in U\setminus\{0\}} \alpha_u(0,y) = 2^k\cdot \left( \sum_{\alpha \in \widehat{Z(G)}} \alpha(0,y) - 1\right) = -2^k
\end{equation}
whenever $y \neq 0$. This completes the proof of \eqref{eq:IHP5}.
\end{proof}

\begin{rem} \label{rem:justHat}
The onus for creating a hyperdifference set in Theorem~\ref{thm:IHP} lies entirely on the antisymmetric map $\widehat{B}$, and not necessarily on $B$ itself. It can and sometimes does happen that a different bilinear map $B_0\colon U \times U \to V$ has $\widehat{B_0} = \widehat{B}$, while $U\times_{B_0} V \ncong U \times_B V$. In that case, the nonisomorphic group $U \times_{B_0} V$ also enjoys a hyperdifference set, which may produce a completely different ETF. For instance, all ten groups of order 64 mentioned in the first row of the table in Subsection \ref{sect:HDSCD} are related in this way.
\end{rem}

\begin{theorem} \label{thm:SuzIHP}
For any odd integer $n = 2k+1 \geq 3$, the $\mathbb{F}_2$-bilinear map $B \colon \mathbb{F}_{2^n} \times \mathbb{F}_{2^n} \to \mathbb{F}_{2^n}$ given by $B(\alpha, \beta) = \alpha \beta^2$ has the injective hyperplane property. Hence, $\mathbb{F}_{2^n}\times_B \mathbb{F}_{2^n}$ admits a hyperdifference set of $2^n - 1$ irreducible representations with constant degree $2^k$.
\end{theorem}

The resulting ETF has size $2^{n-1}(2^n-1) \times 2^{2n}$. The proof of Theorem~\ref{thm:SuzIHP} relies on the field trace $\tr \colon \F_{2^n} \to \F_2$, which is given by
\[ \tr(\alpha) = \alpha + \alpha^2 + \alpha^4 + \dotsb + \alpha^{2^{n-1}} \qquad (\alpha \in \F_{2^n}). \]
For background on this and other matters of finite fields, we refer the reader to \cite{LN83}. The field trace is linear, so it produces a bilinear form $\F_{2^n} \times \F_{2^n} \to \F_2$, called the \emph{trace form}, that maps $(\alpha,\beta) \mapsto \tr(\alpha \beta)$. The trace form is nondegenerate, so the linear functionals on $\F_{2^n}$ are precisely the maps $T_\alpha \colon \F_{2^n} \to \F_2$ given by $T_\alpha(\beta) = \tr(\alpha\beta)$ for $\alpha,\beta \in \F_{2^n}$. The elements of trace zero are precisely those in the subspace
\[ X_1 :=  \{ \beta^2 - \beta \colon \beta \in \F_{2^n} \}. \]
Hence, the hyperplanes in $\F_{2^n}$ are the spaces
\[ X_\alpha := \ker T_{\alpha^{-1}} = \{ \alpha(\beta^2 - \beta) \colon \beta \in \F_{2^n}\}\]
for $\alpha \in \F_{2^n} \setminus \{0\}$, and these are distinct.

\smallskip

\begin{proof}[Proof of Theorem~\ref{thm:SuzIHP}.]
For any $\alpha \in \F_{2^n} \setminus \{0\}$ and any $\beta \in \F_{2^n}$,
\[ \widehat{B}(\alpha,\beta) = \alpha \beta^2 - \alpha^2 \beta = \alpha^3\bigl( (\alpha^{-1} \beta)^2 - (\alpha^{-1} \beta) \bigr). \]
Thus,
\[ \ran L_\alpha = \{ \widehat{B}(\alpha,\beta) \colon \beta \in \F_{2^n} \} = X_{\alpha^3}. \]
To show that $B$ has IHP, we only have to show that $\alpha^3 \neq \beta^3$ when $\alpha \neq \beta$. Equivalently, we need to show that $\alpha^3 \neq 1$ when $\alpha \neq 1$. For this, we must use the fact that $n = 2k+1$ is odd, so that
\[ |\F_{2^n}^\times| = 2^n - 1 = 4^k\cdot 2 - 1 \equiv 1^k\cdot 2 - 1 \equiv 1 \mod 3. \]
Hence, no element of $\F_{2^n}^\times$ has order 3.
\end{proof}

\begin{rem}
After deriving and proving Theorem~\ref{thm:SuzIHP}, the authors learned that the groups it describes were studied as early as 1961 by G.~Higman \cite{Hig63}, who called them examples of \emph{Suzuki 2-groups}. In his terminology, $\mathbb{F}_{2^n} \times_B \mathbb{F}_{2^n} = A(n,\theta)$, where $\theta$ is the Frobenius automorphism $\theta(\alpha) = \alpha^2$ on $\mathbb{F}_{2^n}$. Interestingly, these groups have already caught the attention of the association scheme community as examples of a different phenomenon called \emph{self-duality} \cite{Ban93,Hanaki96,HanOku97}. When $n=3$, the correlation explained in Remark~\ref{rem:justHat} leads to ten non-isomorphic groups of order 64, each of which admits a hyperdifference set with the parameters in Theorem~\ref{thm:IHP}. These ten groups were the basis for our study of equiangular central group frames; they were also identified by Bannai \cite[Thm.\ 5.1]{Ban93} as the first examples of self-dual nonabelian groups. As a curiosity, we mention that every abelian group is self-dual. This means that every group currently known to admit a nontrivial hyperdifference set is self-dual.
\end{rem}

\section{Connections with Heisenberg groups} \label{sec:Heis}

Theorem~\ref{thm:SuzIHP} establishes the existence of a hyperdifference set $D$ for $G = \F_{2^n} \times_B \F_{2^n}$, but it does not explain how to construct it. In this section, we examine $D$ in greater detail, giving an explicit description of both the Gram matrix $\G_D$ and the unitary representations involved in $D$. As we will see, the ETF we produce is essentially made from $2^n - 1$ copies of an expanded Heisenberg group over $\Z_2^k$. The architecture that fits these copies together is controlled by $\F_{2^n}^\times$, disguised as a subgroup of $\Aut(G)$ that spins the expanded Heisenberg group around to build our hyperdifference set.

Let $L^2(\Z_2^k)$ be the Hilbert space of functions $f \colon \Z_2^k \to \C$ with the inner product
\[ \langle f, g \rangle_{L^2(\Z_2^k)} = \sum_{x \in \Z_2^k} f(x) \overline{g(x)} \qquad (f,g \in L^2(\Z_2^k)). \]
Denote $e_0,\dotsc,e_{k-1}$ for the canonical basis of $\Z_2^k$, and give $\Z_2^k$ the usual dot product
\[ (\epsilon_0,\dotsc,\epsilon_{k-1})\cdot (\delta_0,\dotsc,\delta_{k-1}) = \sum_{j=0}^{k-1} \epsilon_j \delta_j \in \Z_2. \]
For $0 \leq s,t \leq k-1$ and $f\in L^2(\Z_2^k)$, we define the translation and modulation $T_s f, M_t f\in L^2(\Z_2^k)$ by $(T_s f)(x) = f(x- e_s)$ and $(M_t f)(x) = (-1)^{x\cdot e_t} f(x)$, respectively. The translation and modulation operators satisfy the power relations
\begin{equation} \label{eq:HeisPow}
T_s^2 = M_t^2 = I
\end{equation}
and the commuting relations
\begin{equation} \label{eq:HeisCom}
T_s T_t = T_t T_s, \qquad M_s M_t = M_t M_s, \qquad \text{and} \qquad T_s M_t = (-1)^{\delta_{s,t}} M_t T_s
\end{equation}
for all $s,t\in\{0,\dotsc, k-1\}$. They generate the \emph{Heisenberg group} over $\Z_2^k$,
\[ \mathbb{H}:= \langle T_s, M_t \colon 0 \leq s,t \leq k-1 \rangle \subset U(L^2(\Z_2^k)). \]
It has order $2^{2k+1} = 2^n$, and its natural representation on $L^2(\Z_2^k)$ is irreducible. (This is well known; for instance, see \cite{Pras09}.)  We will work with the slightly extended group $\mathbb{H}\langle i I \rangle = \mathbb{H} \cup i\mathbb{H}$, which has order $2^{n+1}$. Notably, this group was leveraged in \cite{CCKS} to construct extremal two-angle line sets. 

Our plan is to map $G$ onto $\mathbb{H}\langle i I \rangle$ using a certain basis for the space $X_1 \subset \F_{2^n}$ of trace-zero field elements. Since $1 \notin X_1$ is the unique $\alpha\in \F_{2^n}$ for which $\ker T_\alpha = X_1$, and since $\tr(\alpha^2) = \tr(\alpha)=0$ for all $\alpha \in X_1$, the trace form is nondegenerate and alternating on $X_1$. Consequently, $X_1$ admits a \emph{symplectic basis} $x_0,\dotsc,x_{k-1},\allowbreak y_0,\dotsc,y_{k-1}$, which means that $\tr(x_s x_t) = \tr(y_s y_t) = 0$ and $\tr(x_s y_t) = \delta_{s,t}$ for $0\leq s,t\leq k-1$; see \cite[Chapter~XV,~Section~8]{Lang}. We will deform our symplectic basis with the linear map $\theta\colon \F_{2^n} \to \F_{2^n}$ given by
\begin{equation} \label{eq:theta}
\theta(\alpha) = \alpha^{2^0} + \alpha^{2^2} + \dotsb + \alpha^{2^{n-1}} \qquad (\alpha \in \F_{2^n}).
\end{equation}
Our main result here is the following.

\begin{theorem} \label{thm:HeisRep}
Fix a symplectic basis $x_0,\dotsc,x_{k-1},\allowbreak y_0,\dotsc,y_{k-1} \in X_1$, and let $\alpha_s = \theta(x_s)$ and $\beta_t = \theta(y_t)$ for $0 \leq s,t \leq k-1$.
\begin{enumerate}[(i)]
\item There is a unique irreducible representation $\pi \colon G \to \mathbb{H}\langle i I \rangle$ with
\begin{align*}
\pi(\alpha_s,y) &= (-1)^{\tr(y)}\cdot i^{\tr(\alpha_s^3)} T_s, \\
\pi(\beta_t,y) &= (-1)^{\tr(y)}\cdot i^{\tr(\beta_t^3)} M_t,
\end{align*}
and $\pi(1,y) = (-1)^{\tr(y)}\cdot i I$ for all $y\in \F_{2^n}$ and all $s,t \in \{0,\dotsc,k-1\}$.

\item Each $\gamma\in \F_{2^n}^\times$ defines an automorphism $\psi_{\gamma} \in \Aut(G)$ with $\psi_\gamma(x,y) = (\gamma x, \gamma^3 y)$ for $x,y\in \F_{2^n}$.
\item $D = \{ \chi_\pi \circ \psi_{\gamma^{-1}} \}_{\gamma \in \F_{2^n}^\times}$ is a hyperdifference set for $\mathfrak{X}(G)$.
\end{enumerate}
\end{theorem}

Combining Theorem~\ref{thm:HeisRep} with Theorem~\ref{thm:GrpFrm} gives an explicit description of the ETF associated with $D$. In fact, we have written GAP code that implements this theory and returns a short, fat matrix \cite{GitHub}. The automorphisms $\{ \psi_\gamma \colon \gamma \in \F_{2^n}^\times\}$ form a subgroup of $\Aut(G)$ isomorphic to $\F_{2^n}^\times$, and our hyperdifference set is an orbit of this group under its natural action on $\Irr(G)$. In this sense, the ETF we get has a form suggested by Thill and Hassibi~\cite{TH}.

\begin{rem}
In the spirit of constructivism, let us give one possible description of $\{\alpha_s\}_{s=0}^{k-1}$ and $\{\beta_t\}_{t=0}^{k-1}$. A field element $z\in \F_{2^n}$ generates a \emph{self-dual normal basis} if $\tr(z^{2^i}z^{2^j}) = \delta_{i,j}$ for $0 \leq i,j \leq n-1$. Self-dual normal bases are well studied in the finite field literature, and many constructions are known \cite{Lem75,LemWei88,SerLem80,Wang89}. Once we have such a generator, the reader can check that
\[ x_s = z^{2^{2s}} + z^{2^{2s+1}} \qquad \text{and} \qquad y_t = z^{2^{2t}} + \sum_{j=2t+2}^{n-1} z^{2^{j}} \qquad (0\leq s,t \leq k-1) \]
define a symplectic basis for $X_1$. To compute $\alpha_{s}$ and $\beta_{t}$ we first note the following easily verified formulas
\[\theta(\alpha)+\theta(\alpha^{2}) = \alpha + \tr(\alpha)\quad\text{and}\quad \theta(\alpha)+\theta(\alpha^{4}) = \alpha+\alpha^2.\]
Using these and the observation that $\tr(z^{2^j}) = \tr(z^{2^{j-1}}z^{2^{j-1}}) = 1$ for any $j$, we have
\[\alpha_{s} = \theta(x_{s}) = \theta(z^{2^{2s}}) + \theta\big((z^{2^{2s}})^2\big) = z^{2^{2s}} + \tr(z^{2^{2s}}) = z^{2^{2s}} +1,\]
and
\begin{align*}
\beta_{t} & = \theta(z^{2^{2t}}) + \theta\big((z^{2^{2t}})^4\big) + \sum_{j=2t+3}^{n-1}\theta(z^{2^j}) = z^{2^{2t}} + z^{2^{2t+1}} + \sum_{j=2t+3}^{n-1}\theta(z^{2^j})\\
 & = z^{2^{2t}} + z^{2^{2t+1}} + \sum_{r=t+1}^{k-1}\Big[\theta(z^{2^{2r+1}}) + \theta\big((z^{2^{2r+1}})^{2}\big)\Big] = z^{2^{2t}} + z^{2^{2t+1}} + \sum_{r=t+1}^{k-1} \Big[z^{2^{2r+1}} + \tr(z^{2^{2r+1}})\Big]\\
 & = \begin{cases}
 z^{2^{2t}} + \sum_{j=t}^{k-1} z^{2^{2j+1}} + 1 , & \text{if $k-t$ is even} \\[5 pt]
 z^{2^{2t}} + \sum_{j=t}^{k-1} z^{2^{2j+1}}, & \text{if $k-t$ is odd}
 \end{cases}
\end{align*}
for $0 \leq s,t \leq k-1$.
\end{rem}

The proof of Theorem~\ref{thm:HeisRep} relies on the quotient group $G_1 = \F_{2^n} \times_{\tr \circ B} \F_2$, whose multiplication is given by
\begin{equation} \label{eq:G1Mult}
(\alpha,\epsilon) \cdot (\beta,\delta) = (\alpha + \beta, \epsilon + \delta + \tr(\alpha \beta^2))\qquad (\alpha,\beta \in \F_{2^n};\, \epsilon, \delta \in \F_2).
\end{equation}
For $\alpha, \beta \in \F_{2^n}$, we define
\begin{equation}\label{eq:AnglFrm}
\langle \alpha, \beta \rangle = \tr(\alpha \beta^2 - \alpha^2\beta).
\end{equation}
This is a symmetric bilinear form on $\F_{2^n}$, and in fact $\langle \cdot, \cdot \rangle = (\tr \circ B)\caret(\cdot,\cdot)$. With this notation, we have the following presentation of $G_1$.

\begin{lemma} \label{lem:GensRels}
Let $\gamma_1,\dotsc,\gamma_n$ be a basis for $\F_{2^n}$ over $\F_2$. For $i=1,\dots,n$, define $f_i = (\gamma_i,0) \in G_1$; also define $f_{n+1} = (0,1) \in G_1$. Then $f_1,\dotsc,f_{n+1}$ generate $G_1$, and they satisfy the power relations
\[ P = \{ f_1^2 = f_{n+1}^{\tr(\gamma_1^3)},\dotsc, f_n^2 = f_{n+1}^{\tr(\gamma_n^3)},\, f_{n+1}^2 = 1_{G_1} \} \]
and the commuting relations
\[ C = \{ f_i f_j = f_j f_i f_{n+1}^{\langle \gamma_i, \gamma_j \rangle} \colon 1 \leq i,j \leq n\} \cup \{ f_i f_{n+1} = f_{n+1} f_i \colon 1 \leq i \leq n \}. \]
In fact, $G_1 \cong \langle f_1,\dotsc,f_{n+1} \mid P \cup C \rangle$.
\end{lemma}

\begin{proof}
It is easy to check that $f_1,\dotsc,f_{n+1}$ satisfy $P$ and $C$ using \eqref{eq:G1Mult}; we leave this to the reader. To see that $f_1,\dotsc,f_{n+1}$ generate $G_1$, let $\alpha \in \F_{2^n}$ be arbitrary, and write
\[ \alpha = \sum_{i=1}^n \epsilon_i \gamma_i \]
for some $\epsilon_i \in \F_2$. Then
\[ \prod_{i=1}^n f_i^{\epsilon_i} = (\epsilon_1 \gamma_1,0)\dotsb (\epsilon_n \gamma_n, 0) = (\alpha,\delta) \]
for some $\delta \in \F_2$. Hence,
\[ \{(\alpha,0), (\alpha,1)\} = \left\{ \prod_{i=1}^n f_i^{\epsilon_i},\ f_{n+1}\cdot \prod_{i=1}^n f_i^{\epsilon_i} \right\} \subset \langle f_1,\dotsc,f_{n+1} \rangle. \]
As $\alpha \in \F_{2^n}$ was arbitrary, we conclude that $f_1,\dotsc,f_{n+1}$ generate $G_1$.

Finally, let $H = \langle f_1,\dotsc, f_{n+1} \mid P \cup C \rangle$. This is a polycyclic presentation (see e.g.\ \cite{HEO}), so every element of $H$ can be written in the form $\prod_{i=1}^{n+1} f_i^{\epsilon_i}$ with $\epsilon_i \in \{0,1\}$. In particular, $|H| \leq 2^{n+1}$. Now, von Dyck's Theorem supplies an epimorphism $H \to G_{1}$, and by comparing orders, we conclude that this surjection is an isomorphism.
\end{proof}

Our plan is to construct an isomorphism $\pi_1 \colon G_1 \cong \mathbb{H}\langle i I \rangle$ and then pull it back through an epimorphism $\varphi_1 \colon G \to G_1$. In order to build $\pi_1$, we need a way to see the symplectic structure on $X_1$ through the bilinear form $\langle \cdot, \cdot \rangle$.

Let $\theta \colon \F_{2^n} \to \F_{2^n}$ be given by \eqref{eq:theta}, and let $\eta \colon \F_{2^n} \to \F_{2^n}$ be the linear map with $\eta(\alpha) = \alpha^2 + \alpha$ for $\alpha \in \F_{2^n}$. These functions are very nearly inverses, in the sense that
\begin{align*}
\eta(\theta(\alpha)) &= \left(\alpha^{2^0} + \alpha^{2^2} + \dotsb + \alpha^{2^{n-1}}\right)^2  + \left(\alpha^{2^0} + \alpha^{2^2} + \dotsb + \alpha^{2^{n-1}}\right) \\
&= \left(\alpha^{2^1} + \alpha^{2^3} + \dotsb + \alpha^{2^n}\right) +  \left(\alpha^{2^0} + \alpha^{2^2} + \dotsb + \alpha^{2^{n-1}}\right) \\
&= \left(\alpha^{2^0} + \dotsb + \alpha^{2^{n-1}}\right) + \alpha^{2^n} \\
&= \alpha + \tr(\alpha)
\end{align*}
and similarly $\theta(\eta(\alpha)) = \alpha + \tr(\alpha)$ for $\alpha \in \F_{2^n}$. Since $n$ is odd, $\tr(1) = 1$, and therefore $\tr(\tr(\alpha)) = \tr(\alpha) = \tr\bigl(\alpha^{2^j}\bigr)$ for all $\alpha \in \F_{2^n}$ and all $j \geq 0$. It follows that $\eta$ and $\theta$ both map $\F_{2^n}$ into $X_1$. With this in mind, the equations above show that $\eta$ and $\theta$ restrict to inverse isomorphisms $X_1 \cong X_1$. These maps serve as bridges between the trace form and $\langle \cdot , \cdot \rangle$: for any $\alpha, \beta \in \F_{2^n}$,
\[ \langle \alpha, \beta \rangle = \tr(\alpha \beta^2 + \alpha^2 \beta) = \tr\left(\eta(\alpha)\cdot \eta(\beta) \right). \]
Thus, when $x,y\in X_1$,
\begin{equation} \label{eq:tr2angl}
\langle \theta(x), \theta(y) \rangle = \tr(xy).
\end{equation}

\begin{lemma} \label{lem:HeisRep}
Let $x_0,\dotsc,x_{k-1},\allowbreak y_0,\dotsc,y_{k-1} \in X_1$ be a symplectic basis for $X_1$, and let $\alpha_s = \theta(x_s)$ and $\beta_t = \theta(y_t)$ for $0\leq s,t, \leq k-1$. There is a unique isomorphism $\pi_1 \colon G_1 \cong \mathbb{H}\langle i I \rangle$ with 
\[ \pi_1(1,0) = i I, \qquad \pi_1(\alpha_s,0) = i^{\tr(\alpha_s^3)} T_s, \qquad \text{and} \qquad \pi_1(\beta_t,0) = i^{\tr(\beta_t^3)} M_t \]
 for $0\leq s,t \leq k-1$.
\end{lemma}

\begin{proof}
First, observe that $\alpha_0,\dotsc,\alpha_{k-1},\allowbreak \beta_0,\dotsc,\beta_{k-1}$ is a basis for $X_1$, since $\theta$ restricts to an isomorphism $X_1 \cong X_1$. Thus, $\alpha_0,\dotsc,\alpha_{k-1},\allowbreak \beta_0,\dotsc,\beta_{k-1},1$ is a basis for $\F_{2^n}$. In order to more easily use Lemma~\ref{lem:GensRels} we will set
\[\gamma_{i} = \begin{cases} \alpha_{i-1} & i=1,\ldots,k,\\ \beta_{i-k-1} & i=k+1,\ldots,2k,\\ 1 & i=2k+1=n,\end{cases}\]
$f_{i} = (\gamma_{i},0)$ for $1\leq i\leq n$, and $f_{n+1} = (0,1)$.

Set
\[g_{j} = \begin{cases}i^{\tr(\alpha_{j-1}^3)}T_{j-1} & j=1,\ldots,k\\ i^{\tr(\beta_{j-k-1}^3)}M_{j-k-1} & j=k+1,\ldots,2k\\ iI & j=2k+1=n,\\ -I & j=2k+2=n+1.\end{cases}\]
From \eqref{eq:HeisPow} we see that $g_{j}^2 = g_{n+1}^{\tr(\gamma_{j}^{3})}$ for $1\leq j\leq n$, and clearly $g_{n+1}^2 = 1_{\mathbb{H}\langle iI\rangle}$. This shows that $g_{1},\ldots,g_{n+1}$ satisfy the relations $P$ from Lemma \ref{lem:GensRels}.

From \eqref{eq:tr2angl} we see that
\[ \langle 1, \alpha_s \rangle = \langle 1, \beta_t \rangle = \langle \alpha_s, \alpha_t \rangle = \langle \beta_s, \beta_t \rangle = 0 \qquad \text{and} \qquad \langle \alpha_s,\beta_t \rangle = \delta_{s,t} \]
for $0 \leq s,t \leq k-1$. From this and \eqref{eq:HeisCom} we deduce that $g_{i}g_{j} = g_{j}g_{i}g_{n+1}^{\langle \gamma_{i},\gamma_{j}\rangle}$ for $1\leq i,j\leq n$ and $g_{i}g_{n+1} = g_{n+1}g_{i}$ for $1\leq i\leq n$. That is, $g_{1},\ldots,g_{n+1}$ satisfy the relations $C$ from Lemma \ref{lem:GensRels}.

Now von Dyck's Theorem gives the existence of a homomorphism $\pi_1 \colon G_1 \to \mathbb{H}\langle i I \rangle$ such that $\pi_{1}(f_{i}) = g_{i}$ for $1\leq i\leq n+1$. Since the image of $\pi_1$ generates $\mathbb{H}\langle i I \rangle$, and since $| G_1 | = | \mathbb{H} \langle i I \rangle |$, $\pi_1$ must be an isomorphism.
\end{proof}

Now we can prove our main result.

\begin{proof}[Proof of Theorem~\ref{thm:HeisRep}]
Let $\varphi_1 \colon G \to G_1$ be the epimorphism with $\varphi_1(x,y) = (x,\tr(y))$ for $x,y\in\F_{2^n}$, and let $\pi_1 \colon G_1 \to U(L^2(\Z_2^k))$ be the representation given by Lemma~\ref{lem:HeisRep}. Then
\[ \pi_1(x,\epsilon) = \pi_1(x,0)\cdot \pi_1(0,\epsilon) = (-1)^\epsilon\cdot \pi_1(x,0) \]
for all $x \in \F_{2^n}$ and all $\epsilon \in \F_2$, so the pullback $\pi = \pi_1 \circ \varphi_1$ is exactly as we have described in the theorem statement. Moreover, $\alpha_0,\dotsc,\alpha_{k-1},\allowbreak \beta_0,\dotsc,\beta_{k-1},1$ is a basis for $\F_{2^n}$, so the set
\[ \{ (\alpha_s,y),\, (\beta_t,y),\, (1,y) : y \in \F_{2^n} \} \]
generates $G$. Thus, the images we have given uniquely determine $\pi$. Since $\mathbb{H} \subset \pi(G)$ already acts irreducibly on $L^2(\Z_2^k)$, $\pi$ is irreducible. This completes the proof of (i). The proof of (ii) is an easy exercise. 

For (iii), remember that we get the characters in our hyperdifference set by modding out the hyperplanes guaranteed by IHP and pulling back characters of the quotients. In the proof of Theorem~\ref{thm:SuzIHP}, we saw that the hyperplane corresponding to $\gamma \in \F_{2^n}^\times$ is
\[ H_\gamma = X_{\gamma^3} = \{ x \in \F_{2^n} \colon \tr(\gamma^{-3} x) = 0\}. \]
The map we called $q_\gamma \colon \F_{2^n} \to \F_2$ in the lead-up to Theorem~\ref{thm:IHP} is the same as the one we called $T_{\gamma^{-3}}$ in the discussion after Theorem~\ref{thm:SuzIHP}; it has $q_\gamma(y) = \tr(\gamma^{-3} x)$ for $x \in \F_{2^n}$. Therefore, multiplication in 
\[ G_\gamma = \F_{2^n} \times_{q_\gamma \circ B} \F_2 \cong G / (\{0\} \times H_\gamma) \]
is given by
\[ (x,\epsilon) \cdot (y,\delta) = (x+y,\epsilon + \delta + \tr(\gamma^{-3} x y^2)) \qquad (x,y \in \F_{2^n};\, \epsilon, \delta \in \F_2), \]
and the epimorphism $\varphi_\gamma \colon G \to G_\gamma$ has the neat formula 
\begin{equation}\label{eq:EpAlpha}
\varphi_\gamma(x,y) = (x,\tr(\gamma^{-3} y)) \qquad (x,y \in \F_{2^n}).
\end{equation}
The reader can check that ${\psi_\gamma(\{0\} \times H_1)} = {\{0\} \times H_\gamma}$. Thus, $\psi_\gamma$ factors to give an isomorphism $\widetilde{\psi}_\gamma \colon G_1 \to G_\gamma$; in particular, $\widetilde{\psi}_\gamma(x,\epsilon) = (\gamma x, \epsilon)$ for $x\in \F_{2^n}$ and $\epsilon \in \F_2$. The relationship between these functions is summarized below.
\[ \xymatrix{
G \ar[r]^{\sim}_{\psi_\gamma} \ar@{->>}[d]_{\varphi_1} & G \ar@{->>}[d]^{\varphi_\gamma} \\
G_1 \ar[r]^{\sim}_{\widetilde{\psi}_\gamma} & G_\gamma
} \]
Now $\chi_\gamma := \chi_{\pi_1} \circ \widetilde{\psi}_\gamma^{-1}$ is an irreducible character of $G_\gamma$, and its pullback to $G$ is
\begin{equation} \label{eq:PullAut}
\chi_\gamma \circ \varphi_\gamma = \chi_{\pi_1} \circ \widetilde{\psi}_\gamma^{-1} \circ \varphi_\gamma = \chi_{\pi_1} \circ \varphi_1 \circ \psi_{\gamma}^{-1} = \chi_\pi \circ \psi_{\gamma^{-1}}.
\end{equation}
By Theorem~\ref{thm:IHP}, $D = \{ \chi_\gamma \circ \varphi_\gamma \}_{\gamma \in \F_{2^n}^\times}$ is a hyperdifference set for $\mathfrak{X}(G)$.
\end{proof}

We end with an explicit representation of the Gram matrix $\G_D$,
which may be useful, for example, in the estimation of its restricted
isometry constants \cite{BFMW13}. We leave this investigation for future work.

\begin{cor}
Let $D$ be as in Theorem~\ref{thm:HeisRep}. Then $\G_D$ is the $G\times G$ matrix with entries
\[ (\G_D)_{(x,y),(\alpha,\beta)} = \begin{cases}
\frac{1}{2} - \frac{1}{2^{n+1}}, & \text{if }x = \alpha \text{ and } y = \beta \\
- \frac{1}{2^{n+1}}, & \text{if }x = \alpha \text{ and }y \neq \beta \\
\frac{i}{2^{n+1}} \cdot  (-1)^{\tr( (x-\alpha)^{-3} (y - \beta + \alpha^3 + x\alpha^2) )}, & \text{otherwise}
\end{cases} \]
for $x,y,\alpha,\beta \in \F_{2^n}$.
\end{cor}

\begin{proof}
We continue the notation used in the proof of Theorem~\ref{thm:HeisRep}. For each $\gamma \in \F_{2^n}^\times$, let $\widetilde{\chi}_\gamma = \chi_\gamma \circ \varphi_\gamma = \chi_\pi \circ \psi_{\gamma^{-1}}$, as in \eqref{eq:PullAut}. By \eqref{eq:GrpGrm},
\begin{align*}
\G_{(x,y),(\alpha,\beta)} &= \frac{2^k}{2^{2n}} \sum_{\gamma \in \F_{2^n}^\times} \widetilde{\chi}_\gamma \left( (x,y) \cdot (\alpha,\beta)^{-1} \right) \\
&= \frac{2^k}{2^{2n}} \sum_{\gamma \in \F_{2^n}^\times} \widetilde{\chi}_\gamma ( x - \alpha, y - \beta + \alpha^3 + x\alpha^2).
\end{align*}
When $x = \alpha$ and $y = \beta$, we quickly see that the diagonal entries are given by
\[ \G_{(x,y), (x,y)} = \frac{2^k}{2^{2n}}\cdot (2^n - 1)\cdot 2^k = \frac{1}{2} - \frac{1}{2^{n+1}}. \]
If $x = \alpha$ and $y \neq \beta$, then \eqref{eq:IHP6} gives us
\[ \G_{(x,y),(x,\beta)} = \frac{2^k}{2^{2n}} \sum_{\gamma \in \F_{2^n}^\times} \widetilde{\chi}_\gamma(0, y - \beta) = - \frac{2^{2k}}{2^{2n}} = -\frac{1}{2^{n+1}}. \]
Finally, when $x \neq \alpha$, \eqref{eq:IHP2} and \eqref{eq:EpAlpha} show that $\widetilde{\chi}_\gamma(x-\alpha,z) = 0$ for all $z\in \F_{2^n}$ and all $\gamma \neq x-\alpha$. Hence,
\begin{align*}
\G_{(x,y),(\alpha,\beta)} &= \frac{2^k}{2^{2n}}\, \widetilde{\chi}_{x - \alpha}( x - \alpha, y - \beta + \alpha^3 + x \alpha^2 ) \\
&= \frac{2^k}{2^{2n}}\, \chi_\pi\left(1, (x-\alpha)^{-3}(y - \beta + \alpha^3 + x \alpha^2) \right). \\
\end{align*}
Since $\pi(1,z) = (-1)^{\tr(z)}\cdot i I$ for $z \in \F_{2^n}$, we conclude that
\[ \G_{(x,y),(\alpha,\beta)} = \frac{i}{2^{n+1}}\cdot (-1)^{\tr\left((x-\alpha)^{-3}(y - \beta + \alpha^3 + x \alpha^2) \right)}. \qedhere \]
\end{proof}


\section*{Acknowledgments}
The authors thank the anonymous referee for thoughtful recommendations that significantly altered and greatly improved the manuscript.

Part of this work was conducted during the SOFT 2016:\ Summer of Frame Theory workshop at the Air Force Institute of Technology.
The authors thank Nathaniel Hammen for helpful discussions during this workshop.
This work was partially supported by NSF DMS 1321779, ARO W911NF-16-1-0008, AFOSR F4FGA05076J002, an AFOSR Young Investigator Research Program award, and an AFRL Summer Faculty Fellowship Program award.
The views expressed in this article are those of the authors and do not reflect the official policy or position of the United States Air Force, Department of Defense, or the U.S.\ Government.


\bibliographystyle{abbrv}
\bibliography{references}

\vspace{10 pt}

\end{document}